\documentclass[11pt]{article}
\usepackage{mathrsfs}
\usepackage{mathtools}
\usepackage{amssymb}
\usepackage{amsmath}
\usepackage{amsthm}
\usepackage{amscd}
\usepackage{commath}
\usepackage{epstopdf}
\usepackage{enumerate}
\usepackage{stmaryrd}
\usepackage{enumitem}
\usepackage{hyperref}
\usepackage{thmtools}
\usepackage{listings}
\usepackage{tikz}
\usepackage{tikz-cd}
\usepackage{comment}
\usepackage{lipsum}
\usetikzlibrary{matrix,arrows,decorations.pathmorphing}
\usepackage{bm}

\newtheorem{theorem}{Theorem}[section]
\newtheorem{lemma}[theorem]{Lemma}
\newtheorem{proposition}[theorem]{Proposition}

\newtheorem{corollary}[theorem]{Corollary}
\newtheorem*{theorem*}{Theorem}
\newtheorem*{proposition*}{Proposition}

\theoremstyle{definition}
\newtheorem{definition}[theorem]{Definition}

\newtheorem{example}[theorem]{Example}
\newtheorem*{claim}{Claim}

\theoremstyle{remark}
\newtheorem{remark}[theorem]{Remark}

\newcommand{\B}{\mathbb{B}}
\newcommand{\N}{\mathbb{N}}
\newcommand{\Q}{\mathbb{Q}}

\newcommand{\R}{\mathbb{R}}
\newcommand{\C}{\mathbb{C}}

\newcommand{\Ok}{\mathcal{O}}

\newcommand{\Spec}{\mathrm{Spec}}
\newcommand{\Sp}{\mathrm{Sp}}

\newcommand{\oo}{{\circ\circ}}

 \newenvironment{itemize*}
  {\begin{itemize}[topsep=-\parskip+\jot,itemsep=-\parskip-\jot]}
  {\end{itemize}}

\newenvironment{enumerate*}
  {\begin{enumerate}[label=(\alph*),topsep=-\parskip+\jot,itemsep=-\parskip-\jot]}
  {\end{enumerate}}

\newenvironment{enumerate**}
  {\begin{enumerate}[label=(\roman*),topsep=-\parskip+\jot,itemsep=-\parskip-\jot]}
  {\end{enumerate}}

\title{A non-archimedean definable Chow theorem}
\date{\today}
\author{Abhishek Oswal}

\begin{document}
\setlist[description]{font=\normalfont\itshape\textbullet\space}
\maketitle

\begin{abstract}
    Peterzil and Starchenko have proved the following surprising generalization of Chow's theorem: A closed analytic subset of a complex algebraic variety that is definable in an o-minimal structure, is in fact an algebraic subset. In this paper, we prove a non-archimedean analogue of this result.
\end{abstract}

\section{Introduction}
The theory of o-minimality provides a framework for a `tame topology' on $\R$, imposing some strict finiteness properties, and in many instances providing a way to interpolate between the algebraic and analytic topologies. The introduction of the theory of o-minimality to the complex analytic category was initiated in the work of Peterzil--Starchenko \cite{starchenko2010tame} who develop a theory of holomorphic functions and analytic manifolds in the category of definable objects in an o-minimal structure. A celebrated result in their work, is the Definable Chow Theorem \cite[Theorem 5.1]{starchenko2008nonstandard}, which states that a closed analytic subset of a complex algebraic variety that is simultaneously definable in an o-minimal structure is in fact an algebraic subset. 

A natural question that one may ask is whether an analogue of the algebraization result of Peterzil and Starchenko exists in the non-archimedean setting. The main result of this paper provides an affirmative answer to this question.

In the non-archimedean setting, a number of analogues of o-minimality have been studied, all with the broad goal of creating a framework that would isolate a class of subsets satisfying `tame' topological and finiteness properties. From our perspective, an important class of such subsets is furnished by the so-called rigid subanalytic sets developed by Lipshitz \cite{lipshitz1993rigid} and further studied in a series of influential works by Lipshitz and Robinson (\cite{lipshitz2000dimension}, \cite{lipshitz2000model}, \cite{lipshitz2000rings}). The rigid subanalytic sets in a sense form the analogue of the o-minimal structure $\R_\mathrm{an}$ consisting of the collection of restricted subanalytic subsets of $\R^n$. We refer the reader to \autoref{background-rigid-subanalytic} for an overview of the notion of rigid subanalytic sets.

As a first step towards the definable Chow theorem we prove the following strong version of the Riemann extension theorem in the context of rigid subanalytic sets.

\begin{theorem*}[A rigid subanalytic Riemann extension theorem, \autoref{riemann-extension-subanalytic}]
Let $K$ be an algebraically closed field that is complete with respect to a non-trivial, non-archimedean absolute value $|\cdot| : K \rightarrow \R_{\geq 0}$. Suppose $X$ is a separated and reduced rigid analytic space over $K.$ Let $Y \subseteq X$ be a closed analytic subvariety of $X$ that is everywhere of positive codimension. Then any analytic function $f \in \mathcal{O}_X(X\setminus Y)$ whose graph is a locally subanalytic subset of $X(K) \times K$ extends to a meromorphic function on all of $X,$ i.e. $f \in \mathcal{M}(X).$
\end{theorem*}

We also prove an analogue of the Definable Chow theorem in the rigid subanalytic setting.  In fact, we prove this in the setting of what we refer to as `tame structures'. The definition of a tame structure follows very closely the definition of an o-minimal structure. In \cite{lipshitz1996rigid}, Lipshitz--Robinson prove that rigid subanalytic subsets of the one-dimensional unit disk $K^\circ$ are none other than the subsets that are Boolean combinations of disks. Thus, it is natural to consider arbitrary structures on $K^\circ$ such that the definable subsets of the closed one-dimensional unit disk $K^\circ$ are the Boolean combinations of (open or closed) disks. This is (in an imprecise sense) what we refer to as a `tame structure'. It is worth mentioning that our definition of a tame structure is in fact very closely related to the notion of a $C$-minimal field introduced by Macpherson and Steinhorn \cite{macpherson1996variants}. We refer the reader to \autoref{rem:relation-to-c-minimality} for a more detailed remark on how this relates to $C$-minimality.

After going through the preliminary definitions of tame structures, we then prove some basic results in the dimension theory of tame structures that are needed for the proof of the Definable Chow theorem. The two key results are the invariance of dimension under definable bijections and the Theorem of the Boundary, both of which have been previously proved for rigid subanalytic subsets.

\begin{proposition*}[Invariance of dimension under definable bijections, \autoref{definable-bijection-dimension}]
Let $X \subseteq (\Ok_{\C_p})^m$ and $Y \subseteq (\Ok_{\C_p})^n$ be definable sets (in a fixed tame structure) and $f : X \rightarrow Y$ a definable bijection. Then $\dim(X) = \dim(Y).$
\end{proposition*}

\begin{theorem*}[Theorem of the Boundary, \autoref{theorem-of-boundary}]
Let $X \subseteq (\Ok_{\C_p})^m$ be a definable set. Then $\dim(\mathrm{Fr}(X)) < \dim(X),$ where $\mathrm{Fr}(X)$ denotes the frontier of $X$ in $(\Ok_{\C_p})^m$, that is $\mathrm{Fr}(X) = \mathrm{cl}_{(\Ok_{\C_p})^m}(X)\setminus X$.
\end{theorem*}

Next we prove the following theorem which may be viewed as a definable version of a classical theorem of Liouville in complex geometry. 

\begin{proposition*}[A non-archimedean definable Liouville's theorem, \autoref{upgraded-polynomial-lemma-v2}]
Let $X$ be a reduced scheme of finite type over $\C_p$ and denote by $X^\text{an}$ the rigid analytification of $X.$ Let $f \in H^0(X^\text{an},\mathcal{O}_{X^{\text{an}}})$ be a global rigid analytic function on $X^\text{an}$ such that the graph of $f$ viewed as a subset of $X(\C_p) \times \C_p$ is definable. Then $f \in H^0(X,\mathcal{O}_X).$
\end{proposition*}

Finally, we prove the non-archimedean version of the definable Chow theorem.

\begin{theorem*}[The non-archimedean definable Chow theorem, \autoref{definable-chow-theorem}]
Let $V$ be a reduced algebraic variety over $\C_p,$ and let $X \subseteq V^\mathrm{an}$ be a closed analytic subvariety of the rigid analytic variety $V^\mathrm{an}$ associated to $V$, such that $X \subseteq V(\C_p)$ is definable in a tame structure on $\C_p$. Then $X$ is algebraic.
\end{theorem*}

\subsection*{Outline of the paper}

In \autoref{background-rigid-subanalytic} we provide the reader with some background on the theory of rigid subanalytic sets as developed by Lipshitz and Robinson. This section is mostly expository in nature, and the reader may find most of the results presented in the chapter to be proved in the fantastic papers \cite{lipshitz1988isolated}, \cite{lipshitz1993rigid}, \cite{lipshitz1996rigid}, \cite{lipshitz2000rings}, \cite{lipshitz2000model} and \cite{lipshitz2000dimension}. In \autoref{sec:riemann-extension-subanalytic}, we prove the strong version of the Riemann extension theorem in the rigid subanalytic category. We note that this version of the Riemann extension theorem is not used in the proof of the definable Chow theorem. 

In \autoref{ch:tame-structures} we introduce the notion of tame structures, and proceed to develop some preliminary dimension theory in this context. The theorem of the boundary and the invariance of dimensions under definable bijections are proved here. In \autoref{sec:recollections-on-dim-theory-of-rigid-varieties}, we collect some lemmas on the general dimension theory of rigid analytic varieties that shall be used in the proof of the definable Chow theorem. Most of the results in this section should be well-known, nonetheless complete proofs are provided for lack of a coherent reference.

In \autoref{proof}, we proceed to the proof of the non-archimedean definable Chow theorem. 


\subsection*{Acknowledgements}
It is a great pleasure to thank Jacob Tsimerman for several helpful conversations and suggestions and for asking me the question of whether a non-archimedean version of Peterzil and Starchenko's definable Chow theorem could be proved. 

The proof of the non-archimedean definable Chow theorem is based on an alternate proof of Peterzil and Starchenko's theorem. I learned of this alternate proof at a lecture series given by Benjamin Bakker. I am very grateful to him for the inspiring set of lectures that went a long way in helping me understand the work of Peterzil and Starchenko.

I would also like to thank Angelo Vistoli for the suggestion to use faithfully flat descent in order to prove \autoref{rigid-function-algebraic-on-dense-open-implies-algebraic} below. It is also a pleasure to thank Fran\c{c}ois Loeser for pointing me to Haskell and Macpherson's work on $C$-minimality and for a careful reading of an earlier version of the manuscript.

\section{Rigid subanalytic sets and a Riemann extension theorem}\label{background-rigid-subanalytic}
In this section, we provide a brief overview of subanalytic geometry in the non-archimedean setting. Analogous to the real case, one starts by considering sets that are locally described by Boolean combinations of sets of the form $\{\underline{x} : |f(\underline{x})| \leq |g(x)|\}$ where $f, g$ are analytic functions. As in the real case, for such sets to define a reasonable `tame topology' one must restrict the class of analytic functions. In the non-archimedean setting, such a theory has been developed in a series of works by Leonard Lipshitz and Zachary Robinson based on the analytic functions in the `ring of separated power series'.  In the first part of this section, we summarize some of the main results of their works.

In \autoref{sec:riemann-extension-subanalytic}, we prove a strong version of the Riemann extension theorem for rigid subanalytic sets.

\subsection{Rings of Separated Power Series}
 Suppose $K$ is an algebraically closed field complete with respect to a non-trivial, non-archimedean absolute value $| \cdot | : K \rightarrow \R_{\geq 0}$. We denote by $K^\circ$ the valuation ring consisting of power bounded elements of $K$, and $K^\oo$ denotes the maximal ideal of $K^\circ$ consisting of the topologically nilpotent elements of $K$. We denote by $\widetilde{K} := K^\circ/K^\oo$ the residue field of $K$ and $\widetilde{\,} : K^\circ \rightarrow \widetilde{K}$ shall denote the reduction map.  
 
\begin{definition}
A valued subring $B \subseteq K^\circ$ is called a \emph{$B$-ring} if every $x \in B$ with $|x| = 1$ is a unit in $B$.
\end{definition}

\begin{remark}
Every $B$-ring is a local ring with $B \cap K^\oo$ being its unique maximal ideal. 
\end{remark}

\begin{definition}
A $B$-ring $B \subseteq K^\circ$ is said to be \emph{quasi-Noetherian} if every ideal $\mathfrak{a} \subseteq B$ has a `quasi-finite generating set' i.e. a zero-sequence $\{x_i\}_{i \in \N} \subseteq \mathfrak{a}$ such that any element $a \in \mathfrak{a}$ can be written in the form $a = \sum_{i \geq 0} b_i x_i$ for some $b_i \in B$. We note that we are not insisting that every infinite sum of the form $\sum_{i \geq 0} b_i x_i$ also lies in $\mathfrak{a}.$
\end{definition}

\begin{proposition}[Properties of quasi-Noetherian rings] We have the following properties of quasi-Noetherian rings:
\begin{enumerate}
    \item A Noetherian $B$-subring of $K^\circ$ is quasi-Noetherian. 
    \item If $B$ is quasi-Noetherian and $\{a_i\}_{i \in \N} \subseteq K^\circ$ is a zero-sequence then \[B[a_0,a_1,\ldots]_{\{a\in B[a_0,a_1,\ldots]: |a|=1\}}\] is also quasi-Noetherian. 
    \item The completion of a quasi-Noetherian subring $B \subseteq K^\circ$ (with respect to the restriction of the absolute value $|\cdot|$ to $B$) is also a quasi-Noetherian subring of $K^\circ$. 
    \item The value semi-group $|B \setminus \{0\}| \subseteq \R_{>0}$ is a discrete subset of $\R_{>0}.$
\end{enumerate} 
\end{proposition}

\begin{definition}
\begin{enumerate*}
    \item If $R$ is a complete, Hausdorff topological ring whose topology is defined by a system of ideals $\{\mathfrak{a}_i\}_{i \in I}$ we define the \emph{ring of convergent power series} $R\{x_1,\ldots,x_n\}$ with coefficients in $R$ as:\begin{align*}
    R\{x_1,\ldots,x_n\} := \{\sum_{\nu=(\nu_1,\ldots,\nu_n) \in \N^n} a_\nu x_1^{\nu_1}\ldots x_n^{\nu_n} \in &R\llbracket x_1,\ldots,x_n\rrbracket : \\ &\lim_{\nu_1+\ldots+\nu_n \rightarrow \infty}a_\nu = 0\}.
\end{align*} The topology on $R\{x_1,\ldots,x_n\}$ is defined by declaring $\{\mathfrak{a}_i\cdot R\{\underline{x}\}\}_{i \in \N}$ to be a fundamental system of neighbourhoods of $0$. With this topology $R\{x_1,\ldots,x_n\}$ is also a complete, Hausdorff topological ring. 
    \item The \emph{Tate algebra $T_m(K)$ in $m$-variables over $K$} is defined as $T_m(K) := K \otimes_{K^\circ} K^\circ\{x_1,\ldots,x_m\}.$ We equip $T_m(K)$ with the \emph{Gauss norm} which is defined as follows: $\norm{\sum_{i \geq 0}a_i x^i}_\mathrm{Gauss} := \max_i \{|a_i|\}$. The Guass norm is a multiplicative norm on $T_m(K)$ that makes $T_m(K)$ a Banach $K$-algebra.
\end{enumerate*}
\end{definition}

\begin{definition}[Rings of separated power series]
We fix a complete, quasi-Noetherian subring $E \subseteq K^\circ.$ Denote by $\mathfrak{B}$ the following family of complete, quasi-Noetherian subrings of $K^\circ$:\begin{align*}\mathfrak{B} := \{E[a_0,a_1,\ldots]^\wedge_{\{x \in E[a_0,\ldots]: |x| =1\}} : &\text{ where }   \\ &\{a_i\}_{i \geq 0}\subseteq K^\circ \text{ satisfies } \lim |a_i| = 0 \}.
\end{align*}
Define:
\begin{align*}
    S_{m,n}(E,K)^\circ &:= \varinjlim_{B \in \mathfrak{B}} B\{x_1,\ldots,x_m\}\llbracket\rho_1,\ldots,\rho_n\rrbracket,\\
    S_{m,n}(E,K) &:=  S_{m,n}(E,K)^\circ \otimes_{K^\circ} K. 
\end{align*} For an $f \in S_{m,n}(E,K)$ we define its Gauss norm in the usual way; writing \[f = \sum_{\mu \in \N^{m},\nu \in \N^n} b_{\mu,\nu} x_1^{\mu_1}\cdots x_m^{\mu_n} \rho_1^{\nu_1}\cdots\rho_n^{\nu_n}\] we set $\norm{f}_\mathrm{Gauss}:= \sup_{\mu, \nu} |b_{\mu,\nu}| = \max_{\mu, \nu} |b_{\mu,\nu}|.$
\end{definition}
\begin{remark}
\begin{enumerate*}
    \item We call $S_{m,n}(E,K)$ the ring of separated power series over $K$. When $K = \C_p$ for instance, we may choose $E$ to be the completion of the ring of integers of the maximal unramified extension of $\Q_p$ in $\C_p$. We shall often suppress the reference to $E$ and $K$ in the notation for convenience and often refer to $S_{m,n}(E,K)$ as simply $S_{m,n}$. 
    \item Note that $S_{m,0} = T_m(K)$ and that $S_{m,n} \supseteq T_{m+n}(K)$.
\end{enumerate*}
\end{remark}

\begin{theorem}[Lipshitz--Robinson \cite{lipshitz2000rings}]
The rings $S_{m,n}$ have the following properties:
\begin{enumerate}
    \item $S_{m,n}$ is Noetherian, a UFD, and a Jacobson ring of Krull dimension $m+n$.
    \item For every maximal ideal $\mathfrak{m}$ of $S_{m,n}$ the quotient ring $S_{m,n}/\mathfrak{m}$ is an algebraic extension of $K$. Furthermore, there is a bijection \begin{align*}
        \{\mathfrak{n}\in \mathrm{Max}(K[\underline{x},\underline{\rho}]) : |x_i(\mathfrak{n})|\leq 1, |\rho_j(\mathfrak{n})| < 1\}& \longleftrightarrow \mathrm{Max}(S_{m,n})\\
        &\mathfrak{n} \longmapsto \mathfrak{n}\cdot S_{m,n}. 
    \end{align*}
\end{enumerate}
\end{theorem}

\begin{definition}
For a generalised ring of fractions $\varphi : T_m \rightarrow A$ over $T_m$ and an element $f \in S$ we denote by $\Delta(f)$ the set of all its Hasse derivatives $D_\nu(f) := \frac{1}{\nu_1 ! \ldots \nu_m !}\frac{\partial^{|\nu|}f}{\partial x_1^{\nu_1}\cdots \partial x_m^{\nu_m}}$ for all $\nu \in \N^m$. 
\end{definition}

\subsection{Rigid Subanalytic Sets}

\begin{definition}[The language $L$ of mutiplicatively valued rings]
Denote by \[L = (+,\cdot,|\cdot|,0,1; \overline{\cdot},\overline{<}, \overline{0},\overline{1})\] the language of multiplicatively valued rings. Note that $L$ is a two-sorted language, the operations $+,-,\cdot$ and elements $0,1$ refer to corresponding operations and elements of the underlying ring and $\overline{\cdot},\overline{0},\overline{1}$ are the underlying operations and elements on the value group $\cup \{\overline{0}\}$.
\end{definition}

We set $S := \cup_{m,n \in \N} S_{m,n}(E,K)$ and $T := \cup_{m \geq 0} T_m$. Consider any subset $\mathcal{H} \subseteq S$ such that $\Delta(\mathcal{H}) \subseteq \mathcal{H}$. The two main examples of such $\mathcal{H}$ are provided by $\mathcal{H} = S$ or $\mathcal{H} = T$.

We define now the language $L_\mathcal{H}$ introduced by Lipshitz--Robinson \cite{lipshitz2000model} which are used to define subanalytic sets. $L_\mathcal{H}$ is a three-sorted language; the first sort for the closed unit disk $K^\circ$, the second sort for $K^\oo$ the open unit disk and the last sort for the totally ordered value group$\cup \{\overline{0}\}.$ The sort structure is mostly a bookkeeping device; the first sort helps us to keep track of non-strict inequalities of the form $|f| \leq |g|$ whereas the second sort helps us to keep track of strict inequalities. 

\begin{definition}[The language $L_\mathcal{H}$]
The language $L_\mathcal{H}$ is the language obtained by augmenting to the language $L$ defined above, symbols for every function in $\mathcal{H}$; i.e. for every $f \in \mathcal{H}$, if $f \in S_{m,n}$ we add a function symbol to $L_\mathcal{H}$ with arity $m$ for the first sort and $n$ for the second sort. 
Thus, \begin{align*}
    L_\mathcal{H} := (+,\cdot,|\cdot|,0,1, \{f\}_{f \in \mathcal{H}}; \overline{\cdot},\overline{<}, \overline{0},\overline{1}).
\end{align*}
\end{definition}

\begin{definition}[Globally $\mathcal{H}$-semianalytic, locally semianalytic, and $\mathcal{H}$-sub\-analytic sets]\leavevmode
\begin{enumerate*}
    \item For a complete, valued field $F$ over $K$, a subset $X \subseteq (F_{\mathrm{alg}}^\circ)^m$ is said to be \emph{globally $\mathcal{H}$-semianalytic} (resp. \emph{$\mathcal{H}$-subanalytic}) if $X$ is definable by a quantifier-free (resp. existential) $L_\mathcal{H}$-formula, i.e. if there exists a quantifier-free (resp. existential) first-order formula $\phi(x_1,\ldots,x_m)$ such that $(a_1,\ldots,a_m) \in X$ if and only if $F_{\mathrm{alg}} \models \phi(a_1,\ldots,a_m)$.
    \item In the special case that $\mathcal{H} = S$, the $\mathcal{H}$-semianalytic (resp. subanalytic) sets are referred to as the \emph{globally quasi-affinoid semianalytic} (resp. \emph{quasi-affinoid subanalytic}) sets. Similarly, in the case that $\mathcal{H} = T,$ the $\mathcal{H}$-semianalytic (resp. subanalytic) sets are referred to as \emph{affinoid semianalytic} (resp. \emph{affinoid subanalytic}) sets. 
    \item In the special case that $\mathcal{H} = S$, we denote the language $L_\mathcal{H}$ by $L_\mathrm{an}$. Furthermore, in this case we also define the language $L_\mathrm{an}^*$ as follows. $L_\mathrm{an}^*$ is the language where we augment to $L_\mathrm{an}$ function symbols for every function $f : \mathrm{Max}(T_m(K)) \rightarrow K_\mathrm{alg}$ such that there exists a finite cover of $\mathrm{Max}(T_m(K))$ by $R$-subdomains $\mathrm{Max}(T_m(K)) = \cup_{i=1}^l U_l$ and functions $f_i \in \Ok(U_i)$ such that for every $i, f \vert_{U_i}$ agrees with the function represented by $f_i$ on $U_i$. A subset $X \subseteq (F_\mathrm{alg}^\circ)^m$ is said to be \emph{locally semianalytic} if $X$ is defined by a quantifier-free $L_\mathrm{an}^*$-formula.
\end{enumerate*}    
\end{definition}

    The globally $\mathcal{H}$-semianalytic sets are in other words Boolean combinations of sets defined by inequalities among the analytic functions in $\mathcal{H}$. Similarly, the $\mathcal{H}$-subanalytic sets, being defined by \emph{existential} formulas are precisely the sets obtained by coordinate projections of $\mathcal{H}$-semianalytic sets from higher dimensions.
    
    Just as in the real subanalytic setting, one would now ask whether subanalytic sets satisfy basic closure properties. For instance,  are they closed under taking complements, closures? It turns out that they are. Lipshitz-Robinson \cite{lipshitz2000model} prove a quantifier-simplification theorem for the language $L_\mathcal{H}$ (recalled below), which would imply that any \emph{arbitrary} $L_\mathcal{H}$-definable set is also $\mathcal{H}$-subanalytic. Since complements and closures are all first-order definable in $\mathcal{L}_\mathcal{H}$, the required closure properties would then follow.
    
    Lipshitz and Robinson's proof of the quantifier simplification theorem for $L_\mathcal{H}$, is actually obtained as a consequence of a striking quantifier-\emph{elimination} theorem in a slightly expanded language $\mathcal{L}_{\mathcal{E}(\mathcal{H})}$ which we introduce now. 
    The expanded language $\mathcal{L}_{\mathcal{E}(\mathcal{H})}$, roughly speaking contains function symbols for every function that is \emph{existentially} definable from functions in $\mathcal{H}$ (the precise definitions are given below). 
    The need to expand the language $ \mathcal{L}_\mathcal{H}$ to include such functions is reflected in the fact that for an $f \in \mathcal{H}$, the Weierstrass data outputted by the Weierstrass division theorems in the context of the algebras $S_{m,n}$ are only existentially definable over $\mathcal{H}$. 

    We also note that for a generalized ring of fractions $\varphi : T_m \rightarrow A$ over $T_m$ and for an element $f \in A$, the induced analytic function $f  : \mathrm{Dom}(A)(F) \rightarrow F_\mathrm{alg}$ might not necessarily be in $\mathcal{H}$ but is nevertheless existentially definable over $\mathcal{H}$. 

\begin{definition}[Existentially definable analytic functions] \cite[Definition 2.6]{lipshitz2000model}.
Given a complete valued field extension $F$ of $K$, a subset $X \subseteq (F_\mathrm{alg}^\circ)^m$, and a function $f :X \rightarrow F_\mathrm{alg}$, we say that $f$ is \emph{existentially definable from the functions $g_1,\ldots, g_l$} if there exists a quantifier-free formula $\phi$ in the language $L$ of multiplicatively valued rings, such that \begin{align*}
    y = f(x) \iff \exists\, \underline{z},\, \phi(x,y,\underline{z}, g_1(x,y,\underline{z}), \ldots, g_l(x,y,\underline{z})).
\end{align*} 
\end{definition}

\begin{definition}[The expanded language $L_{\mathcal{E}(\mathcal{H})}$]\leavevmode
\begin{enumerate*}
    \item We set $\mathcal{E}(\mathcal{H})$ to consist of all functions $f :\mathrm{Dom}(A)(F)  \rightarrow F_\mathrm{alg}$ for a generalized ring of fractions $\varphi : T_m \rightarrow A$ over $T_m$ and $f \in A$ such that all of its partial derivatives, i.e. all the functions in $\Delta(f)$ are existentially definable from functions in $\mathcal{H}$. 
    \item The language $L_{\mathcal{E}(\mathcal{H})}$ is the three-sorted language obtained by augmenting $L_\mathcal{H}$ with function symbols for every $f \in \mathcal{E}(\mathcal{H}).$
\end{enumerate*}
\end{definition}

\begin{theorem}[The uniform quantifier elimination theorem of Lipshitz and Robinson \cite{lipshitz2000model}]\label{uniform-quant-elimination}
Fix a subset $\mathcal{H} \subseteq S$ such that $\Delta(\mathcal{H}) = \mathcal{H}$. Let $\varphi(\underline{x})$ be an $L_{\mathcal{E}(\mathcal{H})}$-formula. Then there exists a quantifier-free $L_{\mathcal{E}(\mathcal{H})}$-formula $\psi(\underline{x})$ such that for every complete valued field extension $F$ of $K$ we have that \[F_\mathrm{alg} \models \left(\forall \underline{x}, \varphi(\underline{x})\iff \psi(\underline{x})\right). \]
\end{theorem}

\begin{corollary}[Quantifier simplification for $L_\mathcal{H}$]
For every $L_\mathcal{H}$-formula $\varphi(\underline{x})$, there exists an \emph{existential} $L_\mathcal{H}$-formula $\psi(\underline{x})$ such that for every complete valued field $F$ extending $K$ we have that  \[F_\mathrm{alg} \models \forall \underline{x},\, \varphi(\underline{x}) \iff \psi(\underline{x}).\]
In other words, every $L_\mathcal{H}$-definable subset is in fact $\mathcal{H}$-subanalytic. In particular, the closures and complements of $\mathcal{H}$-subanalytic sets are again $\mathcal{H}$-subanalytic.   
\end{corollary}

\subsection{A rigid subanalytic Riemann extension theorem}\label{sec:riemann-extension-subanalytic}

In this section we prove a version of the Riemann extension theorem in the setting of rigid subanalytic sets. 

Throughout this section and in everything that follows, by subanalytic (without further qualification) we shall simply mean quasi-affinoid subanalytic, i.e. $\mathcal{H}$-subanalytic with $\mathcal{H} = S = \cup_{m,n} S_{m,n}(E,K).$ We also assume in this section that $K$ is algebraically closed. We shall denote by $\B^d$ the $d$-dimensional rigid analytic closed unit disk over $K$, that is $\B^d = \Sp(T_d(K))$.

It is convenient to extend the notion of subanalytic sets to subsets of $K^n$. We make the following definition:

\begin{definition}\label{def:subanalytic*-in-C_p}
A subset $\mathcal{S}\subseteq K^n$ is said to be a subanalytic subset of $K^n$ if the following equivalent conditions are satisfied: 
\begin{enumerate**}
    \item \label{subanalytic*-in-C_p-proj} $\pi_n^{-1}(\mathcal{S})\subseteq (K^\circ)^{n+1}$ is subanalytic, where \begin{gather*}
        \pi_n : (K^\circ)^{n+1}\setminus\{0\} \rightarrow \mathbb{P}^n(K^\circ)=\mathbb{P}^n(K)
    \end{gather*} is the map sending $(z_0,z_1,\ldots,z_n) \mapsto [z_0:z_1:\ldots:z_n].$ 
    
    We view $K^n \subseteq \mathbb{P}^n(K)$ via the map $(z_0,z_1,\ldots,z_{n-1})\mapsto [z_0:z_1:\ldots:z_{n-1}:1].$
    
    \item \label{subanalytic*-in-C_p-balls} For every map $\epsilon :  \{1,2,\ldots,n\} \rightarrow \{\pm 1\}$ the set \begin{gather*}
        \mathcal{T}_\epsilon := \{(\alpha_1,\ldots,\alpha_n) \in (K^\circ)^n : \text{ if } \epsilon(r)=-1 , \alpha_r \neq 0,\\ \text{ and } (\alpha_1^{\epsilon(1)},\ldots,\alpha_i^{\epsilon(i)},\ldots,\alpha_n^{\epsilon(n)} )\in \mathcal{S}\}
    \end{gather*} is a subanalytic subset of $(K^\circ)^n.$
\end{enumerate**}
\end{definition}

It follows that the collection of subanalytic subsets of $K^n$ forms a Boolean algebra of subsets, closed under projections, and moreover forms a structure on $K$ in the sense of \cite[Ch 1, (2.1)]{van2000tame}.

\begin{definition}\label{subanalytic*-in-rigid-space}
Let $X$ be a separated rigid analytic variety over $K$ and let $S \subseteq X$ be a subset. Then we say that $S$ is \emph{locally subanalytic} in $X$ if there exists an admissible cover by admissible affinoid opens $X = \cup_i X_i$ and closed immersions $\beta_i : X_i \hookrightarrow \mathbb{B}^{d_i}$ such that for all $i$, $\beta_i(S \cap X_i)$ is subanalytic in $(K^\circ)^{d_i}.$ 
\end{definition}

It is easy to see that if it
is true for one admissible affinoid cover and some choice of embeddings $\beta_i$, then it's true for any other such cover and embeddings.

\begin{definition}\label{subanalytic*-in-varieties}
Let $V/K$ be a finite-type reduced scheme over $K.$ We say that a subset $S \subseteq V(K)$ is subanalytic if there exists a finite affine open cover $V = \cup_i U_i = \cup_i \Spec(A_i)$ and closed embeddings $U_i(K) \xhookrightarrow{\beta_i} K^{n_i}$ (arising from a presentation of $A_i$ as a quotient of $K[t_1,\ldots,t_{n_i}]$) such that for all $i$, $\beta_i(S\cap U_i(K))$ is subanalytic. 
\end{definition}

\begin{remark}
We note that if $S \subseteq V(K),$ is subanalytic, then for \emph{every} finite affine open cover $U_i$ of $V$ and for any choice of presentations $\beta_i : K[t_1,\ldots,t_{n_i}]\twoheadrightarrow \Ok({U_i}),$ we have that $\beta_i(S \cap U_i(K)) \subseteq K^{n_i}$ is subanalytic.
\end{remark}

\begin{remark}\label{rem:incompatibility-with-analytification*}
Suppose $V$ is a separated finite type scheme over $K$ and $V^\mathrm{an}$ is the associated rigid analytic variety, with analytification map $a_V : V^\mathrm{an} \rightarrow V,$ then we note that the map $a_V$ need not necessarily take a locally subanalytic set on $V^\mathrm{an}$ to a subanalytic set of $V(K)$ in the sense of \autoref{subanalytic*-in-varieties}. Indeed, if we consider the affine line $\mathbb{A}^1_{\C_p},$ and the subset $S:= \cup_{n \geq 0} \{z\in \C_p : |p^{-2n}|\leq |z| \leq |p^{-(2n+1)}|\}.$ Then $S$ is not a rigid subanalytic subset of  the algebraic affine line $\mathbb{A}^1(\C_p)$ nevertheless it is a locally subanalytic subset of the analytification $\mathbb{A}^{1,\mathrm{an}}_{\C_p}$.

But if $V$ is proper then locally subanalytic sets of $V^\mathrm{an}$ are indeed subanalytic in $V(K)$.
\end{remark}

\begin{lemma}\label{subanalytic-in-proper-rigid}
Suppose $V$ is a variety over $K.$ Let $V^\mathrm{an}$ denote the associated rigid analytic space over $K,$ with analytification map $a_V : V^\mathrm{an} \rightarrow V.$ Then, if $S \subseteq V(K)$ is subanalytic as in \autoref{subanalytic*-in-varieties} then $a_V^{-1}(S)\subseteq V^\mathrm{an}$ is locally subanalytic as in \autoref{subanalytic*-in-rigid-space}.

Moreover, if $V$ is \emph{proper} over $K$ then the converse holds, i.e. $S \subseteq V(K)$ is subanalytic $\iff a_V^{-1}(S) \subseteq V^\mathrm{an}$ is locally subanalytic.
\end{lemma}
\begin{proof}
This follows from the fact that proper rigid spaces are quasicompact, and in particular, when $V$ is proper, $V^\mathrm{an}$ has an admissible covering by  \emph{finitely many} affinoids.
\end{proof}



We now turn to the proof of the following version of the Riemann extension theorem. 

\begin{theorem}\label{riemann-extension-subanalytic}
Suppose $X$ is a separated and reduced rigid analytic space over the algebraically closed field $K.$ Let $Y \subseteq X$ be a closed analytic subvariety of $X$ that is everywhere of positive codimension. Then any analytic function $f \in \mathcal{O}_X(X\setminus Y)$ whose graph is a locally subanalytic subset of $X(K) \times K$ extends to a meromorphic function on all of $X,$ i.e. $f \in \mathcal{M}(X).$
\end{theorem}

\subsection*{Outline of the proof}
The proof is inspired by L\"utkebohmert's proof of the usual non-archim\-edean Riemann extension theorem \cite{lutkebohmert1974satz}.
We make a series of reductions in the course of the proof. We summarize the main reduction steps below.

\begin{description}
    \item[Step 1:] The question of extending $f$ meromorphically along $X$ is local for the $G$-topology of $X$ and thus we may assume that $X = \Sp(A)$ is a reduced affinoid. Further, working over irreducible components of $X$, we also assume that $X = \Sp(A)$ is irreducible and thus that $A$ is an integral domain.
    \item[Step 2:] Choose a Noether normalization $\pi : X \rightarrow \B^d$. We show in \autoref{finite-descent-riemann-extension-lemma} that if we prove our theorem for $\B^d$ and the analytic subset $\pi(Y) \subseteq \B^d$, we can conclude the theorem for $X$. Thus, we may assume $X = \B^d$ is the $d$-dimensional rigid unit disk over $K$.
    \item[Step 3:] Since $\mathrm{Sing}(Y)$ is of codimension at least $2$ in $X$, by the non-archimedean Levi extension theorem \cite[Theorem 4.1]{lutkebohmert1974satz}, it suffices to extend $f$ meromorphically to an $f^* \in \mathcal{M}(X\setminus\mathrm{Sing}(Y))$. Replace $X, Y$ by $X \setminus \mathrm{Sing}(Y), Y \setminus \mathrm{Sing}(Y)$ respectively. Once more using Step 1, we reduce to the case where $Y$ is regular/smooth and $X$ is an affinoid subdomain of $\B^d$. 
    \item[Step 4:] Since $X$ and $Y$ are smooth over the algebraically closed field $K$, we are now in a position to use a result of Kiehl (recalled below, \autoref{kiehl's-tubular-neighbourhood}) which tells us that \emph{locally} $Y \subseteq X$ looks like $Z \times \{0\} \subseteq Z \times \B^n$ for a smooth affinoid space $Z$. We may even assume that $n=1$ since if $Y$ is codimension at least $2$, the result we seek is a special case of the non-archimedean Levi extension theorem. In all we are down to the case where $X = Z \times \B^1$ and $Y = Z \times \{0\}$ for a smooth, reduced affinoid space $Z$ over $K$.
    \item[Step 5:] This final case is proved separately in \autoref{riemann-extension-product-lemma}. 
\end{description}

We first recall Kiehl's tubular neighbourhod result. We need the following definition.

\begin{definition}(\cite[Definition 1.11]{kiehl1967derham}.)
We say that an affinoid algebra $A$ over the non-trivially valued non-archimedean field $k$ is \emph{absolutely regular at a maximal ideal} $x$ of $A$ if for every complete valued field $K$ extending $k$ and for every maximal ideal $y$ of $A \widehat{\otimes}_k K$ above $x$, the localization $(A \widehat{\otimes}_k K)_y$ is a regular local ring. If the affinoid algebra $A$ over $k$ is absolutely regular at every one of its maximal ideals we say that $A$ is absolutely regular.
\end{definition}

\begin{remark}
For a maximal ideal $x$ of an affinoid algebra $A$ over an \emph{algebraically closed} (or more generally perfect) non-archimedean field $k$, $A$ is absolutely regular at $x$ if and only if the localization $A_x$ is a regular local ring. 
\end{remark}

\begin{theorem}\label{kiehl's-tubular-neighbourhood}\emph{(Kiehl's tubular neighbourhood theorem, \cite[Theorem 1.18]{kiehl1967derham}).} Suppose $A$ is an affinoid algebra over a non-trivially valued non-archimedean field $k$ and let $\mathfrak{a}$ be an ideal of $A$ generated by $f_1, \ldots, f_l \in A$. Suppose that the quotient affinoid algebra $A/\mathfrak{a}$ is absolutely regular and that $A$ is absolutely regular at every point of $V(\mathfrak{a})$. Then there exists an $\epsilon \in k^\times$ such that the `$\epsilon$-tube' around $V(\mathfrak{a})$,  \[\Sp(B) := \{x \in \Sp(A) : |f_j(x)| \leq |\epsilon|, \forall j=1,\ldots,l\}\] has an admissible affinoid covering $(\Sp(B_i)\rightarrow\Sp(B)), i= 1,\ldots, r$ along with isomorphisms 
\[\phi_i : (B_i/\mathfrak{a}B_i)\{x_1,\ldots,x_{n_i}\} \xrightarrow{\cong} B_i\]
from the free affinoid algebra over $B_i/\mathfrak{a}B_i$ in the variables $x_1,\ldots, x_{n_i},$ such that the elements $\phi_{i}(X_1),\ldots, \phi_i(x_{n_i})$ generate the ideal $\mathfrak{a}B_i.$ 
\end{theorem}

 We recall a result on the number of zeroes of a convergent power series in one variable that shall be used in our proof and then prove the Lemma that allows us to make the reduction mentioned in Step 2.  

\begin{lemma}\label{zeroes-of-tate-functions-new}
Suppose $f(t) \in K\{ t \}$ is an element of the one-dimensional Tate algebra over $K$. Let $\epsilon(f) := \max\{i \geq 0 : |a_i| = \lVert f \rVert_{\text{Gauss}}\}.$ Then the number of zeroes of $f$ (counting multiplicities) in the closed unit disk $K^\circ$ is at least $\epsilon(f).$
\end{lemma}
\begin{proof}
Note that $f$ is ``$t$-distinguished" of degree $\epsilon(f)$ (see \cite[\S 5.2.1, Definition 1]{bosch1984non}). By the Weierstrass Preparation Theorem for Tate algebras (\cite[\S 5.2.2, Theorem 1]{bosch1984non}) we may write $f = e\cdot \omega$ where $e \in K\{t\}^\times$ and $\omega \in K[t]$ is a polynomial of degree $\epsilon(f),$ and $\omega$ has $\epsilon(f)$ zeroes (counting multiplicities) in $K^\circ.$ 
\end{proof}

\begin{lemma}\label{finite-descent-riemann-extension-lemma}
Let $\pi : X \rightarrow S$ be a finite morphism of reduced, irreducible affinoids over $K$ of the same dimension. Suppose $Y \subseteq X$ is a closed analytic subvariety of $X$. Let $T := \pi(Y).$ Suppose that every analytic function $g \in \Ok_S(S\setminus T)$ extends uniquely to a meromorphic function $g^* \in \mathcal{M}(S)$ on $S.$ Then every analytic function $f \in \Ok_X(X \setminus Y)$ extends to a meromorphic function on $X.$  
\end{lemma}
\begin{proof}
Following the proof of \cite[Satz 1.7]{lutkebohmert1974satz}, we see that for any $f \in \Ok_X(X\setminus Y)$ there is a meromorphic $f^* \in \mathcal{M}(X)$ such that $f^*\vert_{X\setminus\pi^{-1}(T)} = f \vert_{X\setminus\pi^{-1}(T)}.$ However, two meromorphic functions that agree on the complement of a positive codimensional analytic subvariety must agree everywhere (see \cite[Lemma 1.1]{lutkebohmert1974satz}). It thus follows that $f^* \vert_{X\setminus Y} = f.$
\end{proof}

We are now equipped to fully prove \autoref{riemann-extension-subanalytic}.

\subsection*{Proof of the rigid subanalytic Riemann extension theorem, \autoref{riemann-extension-subanalytic}.}

\begin{proof}
We first reduce to the case where $X$ is a reduced affinoid space over $K$. Indeed, to make this reduction consider an admissible covering of $X$ by affinoid subdomains $X = \cup_{i \in I} U_i.$ For each $i \in I$, $U_i \cap Y$ is an analytic subvariety of $U_i$ of positive codimension at every point of $U_i$ and furthermore $f \vert_{U_i \setminus Y}$ has a locally subanalytic graph in $U_i(K) \times K$. Suppose that we were able to find for every $i$, meromorphic functions $f_i^* \in \mathcal{M}(U_i)$ such that $f_i^* \vert_{U_i\setminus Y} = f \vert_{U_i \setminus Y}$. Then we note that for any $i, j \in I$ the meromorphic functions $f_i^*\vert_{U_i \cap U_j}$ and $f_j^*\vert_{U_i \cap U_j}$ agree on the complement of the positive codimensional subvariety $Y \cap U_i \cap U_j$ and hence by \cite[Lemma 1.1]{lutkebohmert1974satz}, agree on $U_i \cap U_j$. Thus, the $\{f_i^*\}_{i \in I}$ glue to a global meromorphic function $f^*$ on $X$ extending $f$. We may thus assume henceforth that $X = \Sp(A)$ is a reduced affinoid space over $K$.  

By working on the irreducible components of $X$ we may assume that $X$ is irreducible, and hence that $A$ is an integral domain. Choose a Noether normalization for $X$, i.e. a finite surjective morphism $\pi : X \rightarrow \B^d$ where $d = \dim(X),$ and with the help of \autoref{finite-descent-riemann-extension-lemma} we further assume that $X = \B^d$ is the $d$-dimensional unit disk over $K$. 

Let $f \in \Ok(\B^d\setminus Y)$ be an analytic function such that its graph is locally subanalytic. In order to show that $f$ extends meromorphically to $X$, it suffices to show that $f$ extends to an $f^*\in \mathcal{M}(\B^d \setminus \mathrm{Sing}(Y))$ such that $f^* \vert_{\B^d \setminus Y} = f $. Indeed, since $\mathrm{Sing}(Y)$ is an analytic subset of codimension at least 2 in $\B^d,$ we have an isomorphism $\mathcal{M}(\B^d) \xrightarrow{\cong} \mathcal{M}(\B^d\setminus \mathrm{Sing}(Y))$ by the non-archimedean Levi extension theorem \cite[Theorem 4.1]{lutkebohmert1974satz}. Consider an admissible affinoid covering $\B^d\setminus \mathrm{Sing}(Y) = \cup_i U_i.$ We remark that affinoid subdomains being finite unions of rational subdomains are indeed rigid subanalytic sets and hence $f \vert_{U_i\setminus Y}$ has a locally subanalytic graph. Using \cite[Lemma 1.1]{lutkebohmert1974satz}, we may work individually over each $U_i$, i.e. we are reduced to proving the theorem in the situation where
$X = \Sp(A)$ is an affinoid subdomain of $\B^d$ and $Y \subseteq X$ is a \emph{regular} analytic subvariety of $X$ of positive codimension everywhere. 

Applying the `tubular neighbourhood' result of Kiehl \cite[Theorem 1.18]{kiehl1967derham}, we obtain an admissible covering $(\Sp(B_i)\rightarrow\Sp(B)), i = 1,\ldots, l$ of some `$\epsilon$-tube' $\Sp(B)$ around $Y$ in $X = \Sp(A).$ It suffices now to prove that for every $i = 1,\ldots, l$, $f\vert_{\Sp(B_i)\setminus Y}$ extends to a meromorphic function $f_i^* \in \mathcal{M}(\Sp(B_i))$. Indeed, the $f_i^*$ must necessarily glue to a meromorphic function $f^* \in \mathcal{M}(\Sp(B))$ (using \cite[Lemma 1.1]{lutkebohmert1974satz}) such that $f^* \vert_{\Sp(B)\setminus Y} = f \vert_{\Sp(B) \setminus Y}$. Since the functions $f^* \in \mathcal{M}(\Sp(B))$ and $f \in \Ok(X\setminus Y)$ agree on the intersection $\Sp(B)\setminus Y$ and noting that $\Sp(B) \cup (X\setminus Y) = X$ is an \emph{admissible} open cover of $X = \Sp(A),$ the sections $f^*$ and $f$ glue to a global meromorphic function on $X$.

We are thus reduced to proving the theorem in the case that $X = \Sp(B_i/\mathfrak{a}B_i) \times \B^{n_i}$ and $Y = \Sp(B_i/\mathfrak{a}B_i) \times \{\underline{0}\}$. If $n_i \geq 2,$ then the codimension of $Y$ is at least $2$, and in this case the theorem follows as a special case of the non-archimedean Levi extension theorem \cite[Theorem 4.1]{lutkebohmert1974satz}. Thus we may even assume that $n_i =1$. In all, we are reduced to proving the following special case of the theorem in \autoref{riemann-extension-product-lemma}.
\end{proof}

\begin{lemma}\label{riemann-extension-product-lemma}
Suppose $Z = \Sp(A)$ is a reduced, irreducible affinoid space, $X = Z \times \B^1,$ and $Y = Z\times \{0\} \subseteq X.$ Then every analytic function $f \in \Ok(X\setminus Y)$ whose graph is a locally subanalytic subset of $X \times K$ extends to a meromorphic function $f^* \in \mathcal{M}(X).$ 
\end{lemma}
\begin{proof}
We denote the coordinate on $\B^1$ as $t$. Let $|\cdot|$ represent the supremum norm on the reduced affinoid $A$. Note that $A$ is a Banach algebra over $K$ when endowed with its supremum norm and the supremum norm is equivalent to any residue norm on $A$ (see \cite[\S 6.2.4, Theorem 1]{bosch1984non}).

We may expand $f(t) = \sum_{i\geq 0} a_i t^i + \sum_{j > 0} b_j t^{-j}$ with $a_i, b_j \in A,$ such that $\lim_i |a_i| = 0$ and for every $R > 0,  \lim_j |b_j|R^{j} = 0.$ Since $\sum_i a_i t^i \in A\{t\}$, we have that $\sum_i a_i t^i$ is a rigid subanalytic function on $Z \times \B^1$, and thus the function $g := \sum_{j > 0} b_j t^{-j}$ defined on $X\setminus Y$ has a graph that is a locally subanalytic subset of $X \times K$.  
In particular, for each $z \in Z,$ the function $g(z, t) = \sum_{j > 0} b_j(z) t^{-j}$ on the punctured disc $\B^1\setminus\{0\}$ is also locally subanalytic. Since discrete subanalytic sets must be finite, we get that for each fixed $z$ either $g(z,t)$ is identically zero on $\B^1 \setminus \{0\}$ or $g(z,t)$ has finitely many zeroes in $\B^1 \setminus \{0\}$. 

Consider $h_z(y) := g(z, y^{-1}) = \sum_{j > 0} b_j(z) y^j$. The growth hypothesis $\forall R \in K^\times , \lim_{j \rightarrow 0} |b_j| |R|^j = 0$ on the $b_j$ implies that $h_z(y) \in K\{R^{-1}y\}$ for every $R \in K^\times$. The number of zeroes of $g(z,t)$ on the annulus $|R|^{-1} \leq |t| \leq 1$ is the number of zeroes of $h_z(y)$ on $1 \leq |y| \leq |R|$. For each $R \in K^\times,$ we set $h_{z,R}(y) := \sum_{j > 0} (b_jR^{-j}) y^j$ so that by \autoref{zeroes-of-tate-functions-new} the number of zeroes of $h_z(y)$ on the closed disk $|y| \leq R$,  is given by $\epsilon(h_{z,R}(y))$. 

Now for $i < j$ if $b_i(z), b_j(z) \neq 0$, then for $R$ large enough $|b_i(z)| R^i \leq |b_j(z)| R^j$ and thus $\epsilon(h_{z,R})\geq j.$ Thus, if $b_j(z) \neq 0$ for infinitely many $j$, $h_z(y)$ has infinitely many zeroes going off to $\infty$ and therefore also $g(z,t)$ has an infinite discrete zero set in $\B^1 \setminus \{0\}$ which as noted above is not possible. Thus, for each $z \in Z$, $b_j(z)$ is eventually $0$.
In other words, $Z = \cup_{m \geq 0} \cap_{j > m} V(b_j)$. If the set $\cap_{j > m} V(b_j)$ is not equal to $Z$ then it is a nowhere dense closed subset of $Z$. By the Baire category theorem, $Z$ cannot be a countable union of nowhere dense closed subsets and therefore for large enough $m$, $\cap_{j > m} V(b_j) = V(\sum_{j > m} (b_j))$ must be equal to $Z$. Since $Z$ is reduced, this means that the $b_j \in A$ are eventually zero. Thus $f$ has a finite order pole along $Y$ and hence extends meromorphically. This completes the proof of the Lemma and thus also of \autoref{riemann-extension-subanalytic}.
\end{proof}

\section{Tame structures}\label{ch:tame-structures}

In this section we introduce the notion of a tame structure. The definition of a tame structure closely follows the definition of an o-minimal structure on $\R$ and is suitably adapted as a generalization of the non-archimedean rigid subanalytic sets discussed in the previous section. Let $K$ be a non-trivially valued non-archimedean field with valuation ring $R$ and totally ordered value group $(\Gamma, <)$. A `structure' on $R$, is going to be a collection of subsets of $R^n$ for every $n \geq 0.$ In fact, it turns out to be convenient to keep track of definable subsets of the value group $\Gamma$ as well. Thus, in this setting a `structure' on $R$, is actually a collection of subsets of $R^m \times \Gamma^n$ for $m, n \geq 0$ that are closed under the natural first-order operations (see \autoref{def:structure-non-archimedean} for the precise conditions). A `tame structure' is then defined to be one where the definable subsets of $R$ are precisely the Boolean combinations of disks of $R$. In \autoref{sec:preliminaries-on-tame-structures}, we provide these preliminary definitions and prove some elementary properties of sets definable in tame structures. 

In \autoref{sec:dimension-theory-of-tame-structures}, we develop the basic dimension theory of sets definable in tame structures. As in the o-minimal setting, the dimension of a non-empty definable set $X \subseteq R^m$ is defined as the largest $d \leq m$ such that for some coordinate projection $\pi :R^m \rightarrow R^d$ we have that the interior of $\pi(X)$ in $R^d$ is non-empty. The two key results we prove in this section are: \begin{itemize*}
    \item the invariance of dimension under definable bijections (\autoref{definable-bijection-dimension}) and
    \item the Theorem of the Boundary, \autoref{theorem-of-boundary} which states that for a definable set $X \subseteq R^m$, $\dim(\mathrm{cl}(X)\setminus X) <\dim(X)$.
\end{itemize*}

The purpose of \autoref{sec:recollections-on-dim-theory-of-rigid-varieties} is to collect together some results in the dimension theory of rigid geometry that are needed for the sequel.
Most importantly, we connect the usual notion of dimension in the rigid analytic setting with the concept of definable dimension of the previous section (\autoref{rigid-dimension-equals-definable-dimension}). We also prove in \autoref{local-ring-is-equidimensional}, a result on the dimensions of local rings of equidimensional rigid varieties. This lemma is used in the course of the proof of the definable Chow theorem.  

\subsection*{Notations and conventions for this section} 
 For a subset $X$ of a topological space $Y$ endowed with the subspace topology, the interior, closure, and frontier of $X$ inside $Y$ are denoted by $\mathrm{int}_Y(X), \mathrm{cl}_Y(X)$ and $\mathrm{Fr}_Y(X)$ respectively. We often omit writing the subscript $Y$ when the ambient topological space is clear from the context. We recall that the frontier of $X$ in $Y$ is defined as $\mathrm{Fr}_Y(X):= \mathrm{cl}_Y(X)\setminus X.$ 
 
$K$ denotes a field complete with respect to a non-trivial non-archimedean absolute value $|\cdot| : K \rightarrow \R_{\geq 0}.$ $R$ denotes the valuation ring of $K,$ $\Gamma^\times := |K|$ the value group of $K,$ and $\Gamma := \Gamma^\times \cup \{0\}.$ We choose a pseudo-uniformizer $\varpi \in K^\times,$ i.e. a non-zero element $\varpi \in R$ with $|\varpi| < 1.$

 For an element $\underline{x}=(x_1,\ldots,x_n) \in K^n,$ we set $\norm{\underline{x}} := \max_{1\leq i \leq n} |x_i|.$ For $\underline{x}=(x_1,\ldots,x_n) \in K^n$ and $\underline{r}=(r_1,\ldots,r_n) \in \Gamma^n,$ denote by $\mathbb{D}(\underline{x};\underline{r}) := \{\underline{y}=(y_1,\ldots,y_n) \in K^n : |x_i-y_i| < r_i \text{ for all }i\}$ and let $\overline{\mathbb{D}}(\underline{x};\underline{r}) := \{\underline{y} \in K^n : |x_i-y_i| \leq r_i \text{ for all }i\}.$ The set $\mathbb{D}(\underline{x},\underline{r})$ is referred to as an \emph{open polydisk (or simply open disk) of poyradius $\underline{r}$} and $\overline{\mathbb{D}}(\underline{x},\underline{r})$ as the \emph{closed polydisk/disk of polyradius $\underline{r}$}. 

\emph{We shall assume from now on that $K$ is second countable, i.e. that $K$ has a countable dense subset.} However, when working with the collection of $\mathcal{H}$-subanalytic sets, the hypothesis of second countability can be eliminated from most statements(see \autoref{second-countability-justification}).

\subsection{Preliminaries}\label{sec:preliminaries-on-tame-structures}

\subsubsection{Tame structures}
\begin{definition}\label{def:structure-non-archimedean}
A \emph{structure} on $(R,\Gamma)$ is a collection $\left(\mathfrak{S}_{m,n}\right)_{m, n \geq 0}$ where each $\mathfrak{S}_{m,n}$ is a collection of subsets  of $R^m\times \Gamma^n$ with the following properties:
\begin{enumerate**}
    \item $\mathfrak{S}_{m,n}$ is a Boolean algebra of subsets of $R^m \times \Gamma^n$
    \item If $\mathcal{S} \in \mathfrak{S}_{m,n}$ then $R \times \mathcal{S} \in \mathfrak{S}_{m+1,n}$ and $\mathcal{S} \times \Gamma \in \mathfrak{S}_{m,n+1}$.
    \item The diagonal $\{(x,x): x\in R\} \in \mathfrak{S}_{2,0},$ and similarly $\{(\alpha,\alpha)\in \Gamma^2: \alpha\in \Gamma\} \in \mathfrak{S}_{0,2}.$
    \item If $\mathcal{S} \in \mathfrak{S}_{m,n}$ then $\text{pr}(\mathcal{S}) \in \mathfrak{S}_{m-1,n}$ and $\text{pr}'(\mathcal{S}) \in \mathfrak{S}_{m,n-1},$ where $\text{pr} : R^m \times \Gamma^n \rightarrow R^{m-1} \times \Gamma^n$ denotes the projection forgetting the last $R$ factor and similarly $\text{pr}' : R^m\times \Gamma^n \rightarrow R^m\times \Gamma^{n-1}$ denotes the projection omitting the last $\Gamma$ factor. 
\end{enumerate**}
\end{definition}

\begin{definition}
We say that a structure $(\mathfrak{S}_{m,n})_{m,n\geq 0}$ on $(R,\Gamma)$ is \emph{tame} if 
\begin{itemize*}
    \item $+,\cdot : R^2 \rightarrow R$ are definable i.e. their graphs are in $\mathfrak{S}_{3,0}.$
    \item $| \cdot | : R \rightarrow \Gamma$ is definable i.e. its graph $\{(x,|x|) : x \in R\} \subseteq R\times \Gamma$ is in $\mathfrak{S}_{1,1}$
    \item $\mathfrak{S}_{0,1}$ is the collection of finite unions of (open) intervals and points in the totally ordered abelian group $\Gamma$
    \item $\mathfrak{S}_{1,0}$ is the collection of subsets of $R$ consisting of the Boolean combination of disks (open or closed).
\end{itemize*}
\end{definition}
\begin{remark}
It follows from the axioms that in a tame structure $(R,\Gamma)$, the ordering on $\Gamma$ is also definable, i.e. the set $\{(\lambda, \mu) \in \Gamma^2 : \lambda < \mu\}$ is in $\mathfrak{S}_{0,2}.$
\end{remark}

\begin{remark}[The relation to $C$-minimal structures]\label{rem:relation-to-c-minimality}
We would like to point out that the notion of a tame structure is closely related to the definition of a $C$-minimal field which is a special case of the notion of a $C$-minimal structure. The theory of $C$-minimal structures was introduced in \cite{macpherson1996variants} by Macpherson and Steinhorn (building upon some work by Adeleke--Neumann \cite{adeleke1996primitive}), and has been further developed in the paper \cite{haskell1994cell} by Haskell and Macpherson. A $C$-relation is a ternary relation $C(x,y,z)$, satisfying certain axioms. We refer the reader to the above papers for the precise definitions. From our point of view, the central examples of $C$-minimal structures arise in the context of algebraically closed, non-trivially valued fields. Given such a field $K$ with a (multiplicatively written) non-trivial valuation $|\cdot| : K \rightarrow \Gamma^\times \cup \{0\}$ into a totally ordered abelian group $(\Gamma^\times, 1, \cdot, <)$, there is a natural $C$-relation that one may define:\[C(x,y,z) \iff |x-y| > |y-z|.\] For an expansion $(K,C,0,1,+,-,\cdot,\ldots)$ of the $C$-structure $(K,C)$ to be $C$-minimal it is necessary then that the definable subsets of $K$ (in the expanded language) are precisely the class of Boolean combinations of disks. However it appears that this might not be sufficient to claim that the structure is $C$-minimal, since for $C$-minimality one requires the same property to hold for \emph{every} structure elementarily equivalent to $(K,C,\ldots)$. 
The expansion of an algebraically closed non-trivially valued field with function symbols for elements of its strictly convergent power series rings (or more generally separated power series rings) is in fact a $C$-minimal expansion of the valued field. Thus, the rigid subanalytic sets discussed above are in fact examples of $C$-minimal structures.
In the general context of $C$-minimal structures, Haskell--Macpherson \cite[\S 4]{haskell1994cell} also prove some of the dimension theory results that we prove for tame structures in \autoref{sec:dimension-theory-of-tame-structures}. Nevertheless, we have retained the definition of a tame structure and the following results in their dimension theory to keep the exposition self-contained. Secondly, the proofs we are able to provide in this context are geometric and fairly elementary. Lastly, it appears that the invariance of dimensions under definable bijections is not known in the general setting of $C$-minimal structures or even for general $C$-minimal fields (see the discussion on \cite[p. 159]{haskell1994cell}).
\end{remark}

For the remainder of this section, we fix a tame structure on $(R,\Gamma),$ and definability of sets and maps will be in reference to this fixed structure.

\begin{example}[Rigid subanalytic sets] 
Suppose $K$ is algebraically closed. For such a $K$, the central example of a tame structure shall be those of the rigid subanalytic subsets of Lipshitz \cite{lipshitz2000model} and the $\mathcal{H}$-subanalytic sets defined in \cite{lipshitz2000model}. Indeed, it is proved in \cite{lipshitz1996rigid}, that the subanalytic subsets of $R$ are exactly the Boolean combinations of disks.  
\end{example}

For the sequel it shall also be convenient to talk about definable subsets of $K^n$. We make the following definition:

\begin{definition}\label{def:definable-in-C_p}
We say that a subset $\mathcal{S}\subseteq K^n$ is a definable subset of $K^n$ if the following equivalent conditions are satisfied: 
\begin{enumerate**}
    \item \label{definable-in-C_p-proj} $\pi_n^{-1}(\mathcal{S})\subseteq R^{n+1}$ is definable, where \begin{gather*}
        \pi_n : R^{n+1}\setminus\{0\} \rightarrow \mathbb{P}^n(R)=\mathbb{P}^n(K)
    \end{gather*} is the map sending $(z_0,z_1,\ldots,z_n) \mapsto [z_0:z_1:\ldots:z_n].$ 
    
    We view $K^n \subseteq \mathbb{P}^n(K)$ via the map $(z_0,z_1,\ldots,z_{n-1})\mapsto [z_0:z_1:\ldots:z_{n-1}:1].$
    
    \item \label{definable-in-C_p-balls} For every map $\epsilon :  \{1,2,\ldots,n\} \rightarrow \{\pm 1\}$ the set \begin{gather*}
        \mathcal{T}_\epsilon := \{(\alpha_1,\ldots,\alpha_n) \in R^n : \text{ if } \epsilon(r)=-1 , \alpha_r \neq 0,\\ \text{ and } (\alpha_1^{\epsilon(1)},\ldots,\alpha_i^{\epsilon(i)},\ldots,\alpha_n^{\epsilon(n)} )\in \mathcal{S}\}
    \end{gather*} is a definable subset of $R^n.$
\end{enumerate**}
\end{definition}

It follows that the collection of definable subsets of $K^n$ form a Boolean algebra of subsets, closed under projections, and moreover forms a structure on $K$ in the sense of \cite[Ch 1, (2.1)]{van2000tame}. 



\begin{lemma}[Basic Properties of definable sets and functions]\label{basic-properties}
\begin{enumerate**}
    \item A pol\-ynomial map $\phi : K^n \rightarrow K^m$ is definable (i.e. its graph is a definable subset of $K^{n+m}$). In particular, zero sets of polynomials with $K$-coefficients are definable subsets of $K^n$. 
    \item For definable functions $f , g : K^n \rightarrow K$, the set $\{\underline{\mathbf{z}} \in K^n : |f(\underline{\mathbf{z}})| \leq |g(\underline{\mathbf{z}})|\}$ is a definable subset of $K^n$. 
    \item For a definable function $f : S \rightarrow K$ on a definable subset $S \subseteq K^m,$ we have that $|f(S)| \subseteq \Gamma$ is a finite union of open intervals and points.
    \item \label{partials} Suppose $f : K^n \rightarrow K$ is a definable function that is given by a convergent power series $f(z_1,\ldots,z_n) = \sum_{i \geq 0}a_i(z_1,\ldots,z_{n-1})z_n^i$ then the functions $a_i : K^{n-1}\rightarrow K$ that send $$(z_1,\ldots,z_{n-1})\mapsto a_i(z_1,\ldots,z_{n-1})$$ are also definable.
\end{enumerate**}
\end{lemma}
\begin{proof}
All of these facts follow from the definition of a tame structure. We note in particular that $+ , \cdot : K^2 \rightarrow K$ and $|\cdot| : K \rightarrow \Gamma$ are definable, and that subsets defined by a first-order formula involving definable sets and definable functions must themselves be definable.  
\end{proof}

\begin{definition}\label{definable-in-varieties}
Let $V/K$ be a finite-type reduced scheme over $K.$ We say that a subset $S \subseteq V(K)$ is definable if there exists a finite affine open cover $V = \cup_i U_i = \cup_i \Spec(A_i)$ and closed embeddings $U_i(K) \xhookrightarrow{\beta_i} K^{n_i}$ (arising from a presentation of $A_i$ as a quotient of $K[t_1,\ldots,t_{n_i}]$) such that for all $i$, $\beta_i(S\cap U_i(K))$ is definable. 
\end{definition}

\begin{remark}
We note that if $S \subseteq V(K),$ is definable, then for \emph{every} finite affine open cover $U_i$ of $V$ and for any choice of presentations $\beta_i : K[t_1,\ldots,t_{n_i}]\twoheadrightarrow \Ok({U_i}),$ we have that $\beta_i(S \cap U_i(K)) \subseteq K^{n_i}$ is definable.
\end{remark}

\subsubsection{Dimension Theory of Tame Structures}\label{sec:dimension-theory-of-tame-structures}
Parallel to the notion of definable dimension in o-minimality, in this section, we shall develop the basic dimension theory in the context of tame structures. In particular, we prove the so-called `Theorem of the Boundary' (\autoref{theorem-of-boundary}), which shall be an important input in the proof of the Definable Chow theorem. 

For this section, we shall retain the Notations and Conventions introduced in the previous section. We note that our field $K$ is assumed to be second-countable. Throughout this section, we fix a tame structure on $(R,\Gamma)$ and definability will be with regards to the fixed structure.

We recall the following definition from \cite[Definition 2.1]{lipshitz2000dimension}:

\begin{definition}\label{definable-dimension}
\begin{enumerate*}
    \item For any subset $X \subseteq K^m,$ we define its dimension, denoted $\dim(X)$ as the largest non-negative integer $d \leq m$ such that there exists a collection of $d$ coordinates $I \subseteq \{1,\ldots,m\}$ (with $|I| =d$) such that if $\mathrm{pr}_I : K^m \rightarrow K^d$ denotes the projection to these coordinates, the image $\mathrm{pr}_I(X)$ of $X$ is a subset of $K^d$ with non-empty interior.
    \item For a subset $X \subseteq K^m$ and a point $x \in X,$ the local dimension of $X$ at $x$, denoted $\dim_x(X)$ is defined by:\begin{gather*}
        \dim_x(X) := \min\{\dim(U\cap X) : U\subseteq K^m \text{ is an open containing }x\}
    \end{gather*}
\end{enumerate*}
\end{definition}

\begin{lemma}\label{complement-lemma}
If $X \subseteq R^m$ is definable, then one of $X$ or its complement $X^c$ contains a non-empty open disk of $R^m$.
\end{lemma}
\begin{proof}
Induct on $m$. For $m = 0,1$ this is clear. Let $m \geq 2$ and suppose $X \subseteq R^{m}$ is definable. Consider the projection to the first coordinate $\mathrm{pr} : R^m \rightarrow R.$ For a point $s \in R,$ and for a set $Y \subseteq R^m,$ we denote by $Y_s \subseteq R^{m-1}$ the set $\mathrm{pr}^{-1}(s)\cap Y = \left(\{s\}\times R^{m-1}\right) \cap Y$ viewed as a subset of $R^{m-1}.$ Consider the following two sets:
\begin{align*}
    S_1 &:= \{s \in R : X_s \text{ contains a non-empty disk of } R^{m-1}\}\\
    S_2 &:= \{s \in R : (X^c)_s \text{ contains a non-empty disk of } R^{m-1}\}.
\end{align*}
$S_1, S_2$ are definable. Also, for every fixed $s \in R,$  $R^{m-1} = X_s \cup (X^c)_s$. Thus, by the inductive hypothesis, for every $s$ one of $X_s$ or $X^c_s$ must contain a non-empty disk of $R^{m-1},$ i.e. $R = S_1 \cup S_2.$ By the $m = 1$ case, one of $S_1$ or $S_2$ contains a non-empty disk. Replacing $X$ by $X^c$ if necessary, we may assume without loss of generality, that $S_1$ contains a non-empty open 1-dimensional disk $D \subseteq S_1 \subseteq R$. Recall that $R$ is assumed to be second countable. Let $\{D_i \subseteq R^{m-1}: i \geq 1\}$ be a countable collection of non-empty open disks in $R^{m-1},$ forming a basis of the metric topology of $R^{m-1}.$ For each $i$ define \[T_i := \{s \in D : X_s \supseteq D_i \}\] We have that $\cup_i T_i = D \subseteq R.$ Since definable subsets of $R$ are Boolean combinations of disks, either $T_i$ is finite or $T_i$ has non-empty interior. $K$ being complete, is uncountable and hence, there is some $i$ such that $T_i$ contains a non-empty open disk of $R.$ Say that $T_1$ contains a non-empty open disk $D' \subseteq T_1.$ Then $D' \times D_1 \subseteq X,$ i.e. $X$ contains an $m$-dimensional disk.
\end{proof}

\begin{corollary}\label{interior-closure}
For a definable set $X \subseteq R^m,$ we have \[\mathrm{int}(X) = \varnothing \iff \mathrm{int}(\mathrm{cl}(X)) = \varnothing.\]
\end{corollary}
\begin{proof}
Suppose $\mathrm{int}(X) = \varnothing$, however the closure $\mathrm{cl}(X)$ has non-empty interior. Let $\overline{D} \subseteq \mathrm{cl}(X)$ be a closed disk of $R^m,$ with positive radius. Then $\overline{D}$ is definably homeomorphic to $R^m$ (by scaling the coordinates). So we may apply the above \autoref{complement-lemma} to definable subsets of $\overline{D}$. In particular, since $ X \cap \overline{D},$ has empty interior, by the Lemma $X^c \cap \overline{D}$ contains a non-empty open disk, i.e. $\mathrm{cl}(X)\setminus X$ has non-empty interior, which is impossible.  
\end{proof}

\begin{corollary}\label{countable-interior}
\begin{enumerate}
    \item Suppose $\displaystyle\bigcup_{i = 1}^\infty X_i = R^m$ for a countable collection of definable subsets $X_i.$ Then, there is some $i \geq 1$ such that $\mathrm{int}(X_i) \neq \varnothing.$
    \item For a countable collection $\{X_i\}_{i \geq 1}$ of definable subsets of $R^m,$ we have \[\dim(\cup_i X_i) = \max(\dim(X_i)).\]
\end{enumerate}
\end{corollary}
\begin{proof}
Follows from the Baire Category theorem and \autoref{interior-closure}.
\end{proof}

\begin{corollary}\label{dimension-closure}
For a definable set $X \subseteq R^m,$ we have $\dim(X) = \dim(\mathrm{cl}(X)).$
\end{corollary}
\begin{proof}
Follows from \autoref{interior-closure}.
\end{proof}

\begin{lemma}\label{injective-definable-dimension}
Let $f : R^m \hookrightarrow R^n$ be an injective definable map. Then \[\dim(f(R^m)) \geq m.\]
\end{lemma}
\begin{proof}
Induct on $m$. 

If $m = 1,$ then since $f(R^1)$ is infinite, there must necessarily be some coordinate projection $\mathrm{pr} : R^n \rightarrow R$ such that the image $\mathrm{pr}(f(R^1))$ is infinite. By the tameness axiom, infinite definable sets of $R$ contain non-empty disks. So $\dim(f(R)) \geq 1.$

Now suppose $m \geq 2.$ For any $y \in R,$ denote by $L_y$ the $(m-1)$-dimensional line, $L_y := R^{m-1}\times \{y\}.$ By the inductive hypothesis, we have that $\dim(f(L_y))\geq (m-1).$  So there exists a choice of $(m-1)$-coordinates (depending on $y$) of $R^n,$ such that the projection of $f(L_y)$ to those corresponding coordinates has non-empty interior. For each choice of $I = (i(1),\ldots,i(m-1))$ with $1 \leq i(1) < \ldots < i(m-1) \leq n,$ let $\pi_I : R^n \rightarrow R^{m-1}$ denote the corresponding projection. Let $T_I := \{y \in R : \pi_I(f(L_y)) \text{ contains a non-empty disk }\}.$ Then $\cup_I T_I = R.$ Hence, there is a choice of $I$ such that $T_I$ contains a closed disk of positive radius say $D$. Replacing $R^m$ by $R^{m-1}\times D,$ (and rearranging coordinates if necessary) we may assume that for all $y \in R,$ $\pi(f(L_y))$ contains a non-empty open disk of $R^{m-1}$ where $\pi : R^n \rightarrow R^{m-1}$ is the projection to the first $(m-1)$-coordinates. Enumerate a countable basis $\{B_i : i \geq 1\}$ of non-empty open disks of $R^{m-1}.$ Let $\Lambda_i := \{y \in R : \pi(f(L_y)) \supseteq B_i\}.$ Then $\Lambda_i$ is definable in $R$ and $\cup_{i\geq 1} \Lambda_i = R.$ So there is some $i$ such that $\Lambda_i$ contains a closed disk, say $D',$ of positive radius. Again replacing $R^m$ by $R^{m-1} \times D'$ we may assume that for all $y \in R, \pi(f(L_y)) \supseteq B_0$ where $B_0 \subseteq R^{m-1}$ is some fixed non-empty open disk of $R^{m-1}.$ Let $X := f(R^m) \subseteq R^n.$ For every $b \in B_0,$ and for every $y \in R,$ we have that $X_b \cap f(L_y) \neq \varnothing.$ Thus for every $b \in B_0,$ the set $X_b$ is infinite and so some projection of $X_b$ to the remaining $(n-(m-1))$-coordinates must be infinite and hence contains a non-empty open 1-dimensional disk. For each of these remaining coordinates, $j \in \{m,\ldots, n\}$ let $S_j := \{b \in B_0 : \mathrm{pr}_j(X_b) \text{ contains a non-empty open disk}\}.$ Since $\cup_j S_j = B_0,$ by \autoref{countable-interior} some $S_j$ contains a non-empty open disk. Shrinking $B_0$ further to this smaller disk and rearranging the coordinates if necessary, we may assume that for all $b \in B_0, \mathrm{pr}_m(X_b)$ contains a non-empty open disk of $R$. Enumerate the disks of $R,$ i.e. let $\{C_i :i \geq 1\}$ be a countable basis of non-empty open disks of $R.$ Let $\Gamma_i := \{b \in B_0 : \mathrm{pr}_m(X_b)\supseteq C_i\}.$ Then $\Gamma_i$ are definable and $\cup_i \Gamma_i = B_0.$ By \autoref{countable-interior}, we have an $i$ such that $\Gamma_i$ contains a non-empty open disk say $B'\subseteq \Gamma_i.$ Then $\mathrm{pr}_{(1,\ldots,m)}(X) \supseteq B'\times C_i,$ and therefore $\dim(f(R^m)) = \dim(X) \geq m.$
\end{proof}


\begin{lemma}\label{quasi-finite-graph-decomposition}
Let $X \subseteq R^m$ be definable. Let $d \leq m,$ and let $\mathrm{pr}_{(1,\ldots,d)} : R^m \rightarrow R^d$ denote the projection to the first $d$-coordinates. Suppose that $\mathrm{pr}_{(1,\ldots,d)}(X) = B$ is a closed polydisc in $R^d$ of positive polyradius, such that the restriction of the projection to $X$, $\mathrm{pr}_{(1,\ldots,d)} : X \rightarrow R^d$ is a quasi-finite surjection, with all the fibers having the same size of say $N$ elements. Then there exists a smaller closed polydisk of positive polyradius $B' \subseteq B$ and $N$ definable maps $s_i : B' \rightarrow R^{m-d}, 1\leq i \leq N$ such that $X \cap (B' \times R^{m-d})$ is the disjoint union of the graphs of the $s_i, 1\leq i \leq N.$
\end{lemma} 
\begin{proof}
Induct on $N$. If $N = 1,$ then the projection $\mathrm{pr}_{(1,\ldots,d)} : X \rightarrow B$ is a definable bijection and the Lemma is clear in this case, since $X$ is evidently the graph of the definable inverse of this bijection, composed with the projection to the last $(m-d)$-coordinates. 

Suppose $N \geq 2.$ For every $m \geq 1,$ define \[D_m := \{b \in B : \text{ for all }x_1\neq x_2 \in X_b, \, \norm{x_1 - x_2} \geq |\varpi^m|\}.\] Note that $\cup_{m \geq 1} D_m = R^d,$ and thus by Corollary \ref{countable-interior} for an $m_0, D_{m_0}$ has non-empty interior. Replacing $B$ with a smaller disk in this interior, and shrinking $X$ too, we assume that for all $b \in B,$ and for all $x_1 \neq x_2 \in X_b, \, \norm{x_1-x_2} \geq |\varpi^{m_0}|.$ Now, cover $R^{m-d}$ by countably many non-empty open disks $\{\Delta_j\}_{j \geq 1}$ of polyradius strictly less than $|\varpi^{m_0}|.$ Since $\displaystyle\bigcup_{j \geq 1} \mathrm{pr}_{(1,\ldots,d)}\left((B \times \Delta_j)\cap X\right) = B,$ by Corollary \ref{countable-interior}, for some $j \geq 1, \mathrm{pr}_{(1,\ldots,d)}\left((B \times \Delta_j)\cap X\right)$ has non-empty interior in $R^d.$ We replace $B$ by a smaller closed disk of positive radius contained in this interior. Thus, we now have that $\mathrm{pr}_{(1,\ldots,d)}: (B \times \Delta_j) \cap X \rightarrow B$ is a \emph{bijection} (since distinct points in the fiber of this projection are at least $|\varpi^{m_0}|$ apart in some coordinate, however the polydisc $\Delta_j$ has polyradius $< |\varpi^{m_0}|$ by choice.) Thus the inverse of this bijection, provides a definable section $s:B \hookrightarrow X.$ And letting $s_1 := \mathrm{pr}_{(d+1,\ldots,m)}\circ s,$ we see that the graph of $s_1$ is exactly $(B \times \Delta_j)\cap X.$ Let $Y := X \cap (B \times (R^{m-d}\setminus \Delta_j)).$ Then $\mathrm{pr}_{(1,\ldots,d)} : Y \rightarrow B$ is a quasi-finite surjection with fibers of constant cardinality $N-1.$ We now apply the induction hypothesis to $\mathrm{pr}_{(1,\ldots,d)} : Y \rightarrow B$ to finish the proof of the Lemma.  
\end{proof}

\begin{proposition}[Invariance of dimensions under definable bijections]\label{definable-bijection-dimension}
Let $X \subseteq R^m$ and $Y \subseteq R^n$ be definable sets and $f : X \rightarrow Y$ a definable bijection. Then $\dim(X) = \dim(Y).$
\end{proposition}
\begin{proof}
It suffices to prove that $\dim(X) \leq \dim(Y),$ since the inverse $f^{-1} :Y \rightarrow X$ is also definable. 
Let $d = \dim(X).$ Suppose the projection of $X$ to the first $d$-coordinates contains a $d$-dimensional non-empty open disk $B$, i.e. $\mathrm{pr}_{(1,\ldots,d)}(X) \supseteq B.$ Replacing $X$ by $X \cap (B \times R^{m-d})$ we may assume that $\mathrm{pr}_{(1,\ldots,d)}(X) = B.$

\begin{claim}
\emph{There is a non-empty open disk $B' \subseteq B$ such that the projection map $\mathrm{pr}_{(1,\ldots,d)} :X \rightarrow B$ is quasi-finite over $B'$ with constant fiber cardinality of size $N$.}
\end{claim} 
For each $j \in \{d+1,\ldots,m\},$ let 
\[T_j := \{b \in B : \mathrm{pr}_j(X_b) \text{ contains a non-empty open disk}\}\] and for each natural number $k \geq 1,$ let $F_k :=\{b \in B : |X_b| = k\}.$ Then $B = \displaystyle\bigcup_{k=1}^\infty F_k \cup \displaystyle\bigcup_{d+1\leq j \leq m} T_j.$ If for any $k \geq 1, F_k$ has non-empty interior, the Claim would be proved. So suppose each $F_k$ has empty interior; then by \autoref{countable-interior} there is some $j$ such that $T_j$ has non-empty interior. Suppose w.l.o.g that $T_{d+1}$ contains a non-empty open disk. Replacing $B$ by this smaller disk (and modifying $X$ appropriately), we may assume that $T_{d+1} = B.$ Let $\{B_i : i \geq 1\}$ be an enumeration of a countable basis of non-empty open disks of $R$, and let $\mathcal{K}_i := \{b \in B : \mathrm{pr}_{d+1}(X_b) \supseteq B_i\}.$ Then $\cup_i \mathcal{K}_i = B$ and so by \autoref{countable-interior} there is an $i_o$ such that $\mathcal{K}_{i_0}$ contains a non-empty open disk $D.$ Replacing $B$ by $D$ we assume that for all $b \in B, \mathrm{pr}_{d+1}(X_b) \supseteq B_{i_0}.$ But then $\mathrm{pr}_{(1,\ldots,d+1)}(X)\supseteq B\times B_{i_0}$ contradicting that $\dim(X) = d,$ and thus proving the claim.

Replacing $B$ with $B'$ obtained from the above claim, and replacing $X$ by $X \cap (B' \times R^{m-d})$ we assume $\mathrm{pr}_{(1,\ldots,d)}: X\rightarrow B$ is quasi-finite with constant fiber cardinality of size $N \geq 1$. By \autoref{quasi-finite-graph-decomposition}, (after possibly shrinking $B$ further) we can find a definable section $s:B \hookrightarrow X.$ By \autoref{injective-definable-dimension} now, $\dim(f(s(B)))\geq d$ and as $Y \supseteq f(s(B))$ we get that $\dim(Y) \geq d$ as was needed.
\end{proof}

\begin{lemma}\label{uniform-continuity-over-small-disk}
Let $D \subseteq R^d$ be a closed polydisk
of positive polyradius. Let $s : D \rightarrow R$ be a definable function. Then given any $\epsilon > 0,$ there exists a smaller closed polydisk $D' \subseteq D$ or positive polyradius such that $s(D')$ is contained in a disk of diameter $< \epsilon,$ i.e. for all $x, y \in D', |s(x) - s(y)| < \epsilon.$  
\end{lemma}
\begin{proof}
Cover $R$ by countably many non-empty open disks $\{B_i\}_{i \geq 1}$ each of diameter $< \epsilon.$ Then, $D = \cup_i s^{-1}(B_i).$ By \autoref{countable-interior}, there is an $i \geq 1,$ such that $s^{-1}(B_i)$ has non-empty interior in $R^d.$ For such an $i$, take $D'\subseteq s^{-1}(B_i)$ to be a closed polydisk of positive polyradius.
\end{proof}

\begin{theorem}[Theorem of the Boundary]\label{theorem-of-boundary}
Let $X \subseteq R^m$ be a definable set. Then $\dim(\mathrm{Fr}(X)) < \dim(X).$
\end{theorem}
\begin{proof}
Let $d = \dim(X).$ By \autoref{dimension-closure}, we first note that $\dim(\mathrm{Fr}(X))\leq \dim(\mathrm{cl}(X)) = \dim(X)=d.$ Suppose for the sake of contradiction, that $\dim(\mathrm{Fr}(X)) = d,$ and that the projection of $\mathrm{Fr}(X)$ to the first $d$-coordinates has non-empty interior. Thus if $\pi : R^m \rightarrow R^d,$ denotes the projection to the first $d$-coordinates, there exists a closed polydisk $D$ of positive polyradius in $R^d,$ such that $\pi(\mathrm{Fr}(X))\supseteq D.$ 

So we have $D \subseteq \pi(\mathrm{Fr}(X))\subseteq \pi(\mathrm{cl}(X)) \subseteq \mathrm{cl}(\pi(X)),$ and hence in particular $\mathrm{cl}_D (\pi(X)\cap D) = D.$ By \autoref{dimension-closure}, $\pi(X)\cap D$ contains a smaller closed disk of positive radius say $D'\subseteq \pi(X)\cap D \subseteq D.$ Replacing $D$ with $D'$ and $X$ with $X \cap (D' \times R^{m-d}),$ we may assume that $D = \pi(X) = \pi(\mathrm{Fr}(X)).$ 

We note that in the argument that follows, we shall often replace $D$ with a smaller disk. This is justified since, if $D' \subseteq D$ is a smaller closed disk  of positive radius, replacing $D$ by $D'$  and $X$ by $X \cap (D'\times R^{m-d})$ does not change the property that $D = \pi(X) = \pi(\mathrm{Fr}(X)).$ 

Continuing our proof, for each $j \in \{d+1,\ldots,m\}$ we let \[\Lambda_j := \{b \in D : \pi_j(X_b) \text{ has non-empty interior}\}.\] If for some $j$, $\Lambda_j$ has non-empty interior, then again using the same trick of enumerating a countable basis of disks in $R$, and following the same line of argument we would conclude that $\pi_{(1,\ldots, d,j)}(X)$ contains a $(d+1)$-dimensional disk. This is not possible since $\dim(X) = d.$ Therefore, for each $j \in \{d+1,\ldots, m\}$ we must have that $\mathrm{int}(\Lambda_j) = \varnothing.$ Therefore, by \autoref{countable-interior}, $D\setminus \cup_j \Lambda_j$ contains a closed disk of positive polyradius. Replacing $D$ with this smaller closed disk, we may assume then that $\pi : X \rightarrow D$ is a quasi-finite surjection. Moreover, using the argument that we made in the Proof of the claim in the proof of \autoref{definable-bijection-dimension}, we may assume the fibers of $\pi : X \rightarrow D$ have constant finite cardinality of size $N$. Running the exact same argument for $\mathrm{Fr}(X)$ instead of $X$ and shrinking $D$ if necessary, we also assume that $\pi : \mathrm{Fr}(X)\rightarrow D$ is quasi-finite surjection with fibers of constant size say $M$. 

Further shrinking $D$ to a smaller closed disk, we may assume by \autoref{quasi-finite-graph-decomposition}, that $X$ is the disjoint union of graphs of $N$ definable functions $s_i : D \rightarrow R^{m-d}.$ If we let $T_i$ denote the graph of $s_i$, then since $X = \cup_{1\leq i \leq N} T_i$, we have  $\mathrm{Fr}(X)\subseteq \cup_i \mathrm{Fr}(T_i).$ Furthermore, since $D = \pi(\mathrm{Fr}(X)) = \cup_i \pi(\mathrm{Fr}(T_i)),$ for some $i$ we must have that $\pi(\mathrm{Fr}(T_i))$ has non-empty interior. We may then replace $D$ by a smaller disk in this interior, and $X$ by $T_i$ and assume thus that $X$ is the graph of a definable function $s : D \rightarrow R^{m-d}$. Furthermore, running the argument in the above paragraphs again, we may ensure that the property that $\pi : \mathrm{Fr}(X)\rightarrow D$ is a quasi-finite surjection of constant fiber cardinality still holds. In all, we have therefore reduced to the following situation:

\emph{$X$ is the graph of a definable function $s : D \rightarrow R^{m-d}$ such that $\pi(\mathrm{Fr}(X)) = D = \pi(X),$ and such that $\pi : \mathrm{Fr}(X) \rightarrow D$ is a quasi-finite surjection with constant fiber cardinality of size $M.$}

Applying \autoref{quasi-finite-graph-decomposition} to $\mathrm{Fr}(X),$ we assume that $\mathrm{Fr}(X)$ is the disjoint union of graphs of $M$ definable functions $g_j : D \rightarrow R^{m-d}, 1\leq j \leq M.$ Let $Y_j$ denote the graph of $g_j,$ so that $\mathrm{Fr}(X) = \cup_{j=1}^M Y_j.$ We note that $X \cap Y_j = \varnothing$ for each $j,$ or in other words for each $j$ and for every $b \in D, \norm{s(b) - g_j(b)} \neq 0.$ For every $m \geq 1,$ define $E_m := \{b \in D : \norm{s(b) - g_j(b)} > |\varpi^m| \text{ for each } j\}.$ We have $\cup_{m \geq 1} E_m = D,$ and thus by \autoref{countable-interior} some $E_m$ has non-empty interior. Replacing $D$ by a smaller closed disk contained in this interior, we may assume that there is some $m_0$ large enough, such that for all $1 \leq j \leq M$ and for all $b \in D, \norm{s(b)-g_j(b)} > |\varpi^{m_0}|.$ Applying \autoref{uniform-continuity-over-small-disk}, and shrinking $D$ to a smaller disk, we may assume that for all $x, y \in D, \norm{s(x)-s(y)} < |\varpi^{m_0 }|.$ 

Now choose a $b \in D,$ and consider $y = (b,g_1(b)) \in Y_1.$ Since $y \in \mathrm{Fr}(X)$, in particular $y$ is in the closure of $X,$ and thus there must exist some $b' \in D$ such that $\norm{s(b')-g_1(b)} < |\varpi^{m_0}|.$ However, by our choice of $m_0,$ $\norm{s(b)-g_1(b)} > |\varpi^{m_0}|.$ By the non-archimedean triangle inequality, we therefore get that $\norm{s(b')-s(b)} > |\varpi^{m_0}|,$ contradicting the conclusion of the previous paragraph. 
\end{proof}

\begin{remark}\label{second-countability-justification}
Note that in all our proofs we have extensively made use of the standing assumption that $K$ is second countable. However, when the tame structure under consideration is that of the $\mathcal{H}$-subanalytic sets this assumption can be removed, exploiting the model completeness and \emph{uniform} quantifier elimination results of \cite{lipshitz2000model}. See for example the argument used in the Proof of \cite[Lemma 2.3]{lipshitz2000dimension}. Running the same argument given there, with appropriate modifications, enables us to reduce the proof of the Theorem of the Boundary for $\mathcal{H}$-subanalytic sets to the case where $K$ is second countable.
\end{remark}

\subsection{Miscellaneous lemmas}\label{sec:recollections-on-dim-theory-of-rigid-varieties}


In this subsection we collect a few auxiliary results relating to the dimension theory of rigid analytic spaces. These results shall be used in the following sections. We prove in \autoref{rigid-dimension-equals-definable-dimension} that the usual notion of dimension in rigid geometry defined via Krull dimensions of associated rings of analytic functions agrees with the notion of definable dimension defined above via coordinate projections. In \autoref{local-ring-is-equidimensional}, we show that for any point $x$ of  a reduced, equidimensional rigid variety $X$, every minimal prime ideal of the local ring $\mathcal{O}_{X,x}$ has the same coheight. This result is used in the sequel in the course of proving the definability of the \'etale locus of a certain finite map. While the results of this section should be fairly standard, we provide their proofs for completeness. 

\begin{definition}
If $X$ is a rigid variety over $K$, we define its dimension, denoted $\dim(X)$, by \[\dim(X) := \max_{x\in X}\{\dim(\mathcal{O}_{X,x})\}.\] 
\end{definition}
\begin{lemma}\label{affinoid's-dimension} Let $Y = \Sp(A)$ be a $K$-affinoid space. Let $\{Y_i\}_{1\leq i \leq m}$ denote the finitely many irreducible components of $Y$. Then,
\begin{enumerate}
    \item $\dim(Y) = \dim(A)$ and 
    \item For any point $y \in Y$, $\dim(\mathcal{O}_{Y,y}) = \max\{\dim(Y_j) : y \in Y_j\}.$
\end{enumerate}
\end{lemma}
\begin{proof}
These facts are rather standard. Due to lack of an explicit reference, we provide a proof nevertheless.

\emph{Proof of (1):}
For a point $y \in Y,$ corresponding to $\mathfrak{m} \in \mathrm{MaxSpec}(A),$ we have $\widehat{A}_\mathfrak{m} = \widehat{\mathcal{O}}_{Y,y}$ \cite[\S 7.3.2 Proposition 3]{bosch1984non}. Since the Krull dimension of a Noetherian local ring is preserved under completion  (see \cite[\href{https://stacks.math.columbia.edu/tag/07NV}{Tag 07NV}]{stacks-project}) we get, \begin{align*}
    \dim(A) &= \max_{\mathfrak{m}\in \mathrm{MaxSpec}(A)}\{\dim(A_\mathfrak{m})\} = \max_{\mathfrak{m}\in \mathrm{MaxSpec}(A)}\{\dim(\widehat{A}_\mathfrak{m})\} \\
    &= \max_{y \in Y}\{\dim(\widehat{\mathcal{O}}_{Y,y})\}= \max_{y \in Y}\{\dim({\mathcal{O}}_{Y,y})\} = \dim(Y).
\end{align*}

\emph{Proof of (2):} From the argument above, if $\mathfrak{m}\in \mathrm{MaxSpec}(A),$ corresponds to $y,$ we know that $\dim(\mathcal{O}_{Y,y}) = \dim(A_\mathfrak{m}).$ If the irreducible component $Y_i$ corresponds to the minimal prime $\mathfrak{p}_i \subset A,$ then we note that $\dim(A_\mathfrak{m}) = \max\{\dim(A_\mathfrak{m}/\mathfrak{p}_jA_\mathfrak{m}) : \mathfrak{p}_j \subseteq \mathfrak{m}\} = \max\{\dim\left((A/\mathfrak{p}_j)_\mathfrak{m}\right) : \mathfrak{p}_j \subseteq \mathfrak{m}\}.$ Now, $A/\mathfrak{p}_j$ is an affinoid algebra that is an integral domain, and this implies that $\dim(A/\mathfrak{p}_j)_\mathfrak{m} = \dim(A/\mathfrak{p}_j)$ - see \autoref{integral-rigid-spaces-are-equidimensional}-1. below. Thus, $\dim(\mathcal{O}_{Y,y})=\dim(A_\mathfrak{m}) = \max\{\dim (A/\mathfrak{p}_j) : \mathfrak{p}_j \subseteq \mathfrak{m}\} = \max\{\dim(Y_j) : y \in Y_j\}.$
\end{proof}

\begin{lemma}\label{rigid-dimension-equals-definable-dimension} Suppose $Y = \Sp(A)$ is a $K$-affinoid space. Suppose $\pi : T_n(K) \twoheadrightarrow A$ is a surjective homomorphism of $K$-algebras. Via $\pi$ we may view $i : Y \hookrightarrow R^n$ as a subset of the $n$-dimensional unit ball $R^n.$ Then,
\begin{enumerate}
    \item  The dimension of $i(Y)$ as a subset of $R^n$ (in the sense of \autoref{definable-dimension}) is the same as the dimension of $Y$ as a rigid analytic space.
    \item For a point $y \in Y,$ the local dimension $\dim_{i(y)}i(Y)$ (in the sense of \autoref{definable-dimension}) is equal to $\dim(\mathcal{O}_{Y,y}).$ 
    \item Suppose $X$ is a rigid space over $K$ and $i : X\hookrightarrow
    \mathbb{A}^{n,\mathrm{an}}_{K}$ a closed immersion. Then $\dim(X)$ equals the dimension of $i(X)$ viewed as a subset of $K^n$ (as in \autoref{definable-dimension}). For a point $x \in X,$ the local dimension $\dim_{i(x)}i(X)$ (as in \autoref{definable-dimension}) is equal to $\dim(\mathcal{O}_{X,x}).$ 
\end{enumerate}
\end{lemma}
\begin{proof}[Proof of (1):]
This is a special case of \cite[Lemma 4.2]{lipshitz2000dimension}. Alternatively, we may reduce to the case that $K$ is second-countable (see \autoref{second-countability-justification}). Then, using Noether's normalization for affinoid algebras, if $d = \dim(A),$ we have a quasi-finite subanalytic surjection $i(Y) \twoheadrightarrow R^d.$ And then we may use an argument very similar to that of \autoref{definable-bijection-dimension}. We omit the details for the alternate argument.

\emph{Proof of (2)}: By definition, we have that \[\dim_{i(y)}i(Y) = \min\{\dim(U \cap i(Y)) : U \subseteq R^n \text{ is open with }i(y)\in U\}.\] We may take this minimum instead over all closed polydisks $\overline{\mathbb{D}}$ of $R^n$ of positive polyradius containing $i(y),$ i.e. \[\dim_{i(y)}i(Y) = \min\left\{\dim\left(\overline{\mathbb{D}}(i(y),\underline{r}) \cap i(Y)\right) : \underline{r} > \underline{0}\right\}.\] 
Since $i^{-1}(\overline{\mathbb{D}}(i(y),\underline{r}))$ is an affinoid subdomain of $Y,$ from the first part (1), we have that \[\dim\left(\overline{\mathbb{D}}(i(y),\underline{r}) \cap i(Y)\right) = \dim(i^{-1}(\overline{\mathbb{D}}(i(y),\underline{r})))\] where the dimension on the right side is the dimension of the affinoid subdomain $i^{-1}(\overline{\mathbb{D}}(i(y),\underline{r}))$ as an analytic space.
Furthermore, note that the affinoid subdomains of the form $i^{-1}(\overline{\mathbb{D}}(i(y),\underline{r}))$ are cofinal in the collection of all affinoid subdomains of $Y$ containing $y$
(use for example \cite[Lemma 1.1.4]{conrad1999irreducible} to see this). Therefore, \begin{align*}
    \dim_{i(y)}i(Y) = \min\{\dim(W) : \,&W \subseteq Y \text{ is an affinoid subdomain} \\&\text{containing }y\}.
\end{align*} The right-hand-side is indeed equal to $\dim(\mathcal{O}_{Y,y})$ (follows from \cite[1.17]{ducros2007variation} and \autoref{affinoid's-dimension}-(2)). 

\emph{Proof of (3):} Follows immediately from (1) and (2).
\end{proof}

\begin{lemma}\label{integral-rigid-spaces-are-equidimensional}
\begin{enumerate}
    \item Suppose $A$ is a $K$-affinoid algebra that is an integral domain. Then every maximal ideal of $A$ has the same height.
    \item Suppose $Y$ is an irreducible rigid analytic variety. Then $Y$ is equidimensional, i.e. for all $y \in Y, \dim(\mathcal{O}_{Y,y}) = \dim(Y).$
\end{enumerate}
\end{lemma}
\begin{proof}
For (1), use Noether normalization for affinoid algebras, the Going Down theorem (\cite[\href{https://stacks.math.columbia.edu/tag/00H8}{Tag 00H8}]{stacks-project}) and \cite[Chapter 2, Proposition 17]{bosch1984non}.

For (2), we refer the reader to the paragraph preceding \cite[Lemma 2.2.3]{conrad1999irreducible}.  
\end{proof}

\begin{lemma}\label{local-ring-is-equidimensional}
Let $X$ be a reduced equidimensional rigid space over $K$, i.e. $\dim(\mathcal{O}_{X,x}) = \dim(X)$ for all $x \in X.$ Then for every $x \in X$ and for every minimal prime ideal $\mathfrak{q}$ of $\mathcal{O}_{X,x}$ we have $\dim(\mathcal{O}_{X,x}/\mathfrak{q}) = \dim(X).$
\end{lemma}
\begin{proof}
Evidently for every $x \in X$, $\dim(\mathcal{O}_{X,x}/\mathfrak{q}) \leq \dim(\mathcal{O}_{X,x}) = \dim(X).$ Suppose the Lemma was false. Then for some $x,$ we would have \[\dim(\mathcal{O}_{X,x}/\mathfrak{q}) < \dim(X).\] Since $\mathcal{O}_{X,x}$ is Noetherian (\cite[\S 7.3.2 Proposition 7]{bosch1984non}), $\mathfrak{q}$ is finitely generated; let us suppose that $\mathfrak{q} = (h_1,\ldots,h_m),$ for elements $h_i \in \mathcal{O}_{X,x}.$ We may choose an open affinoid domain $\Sp(B)$ in $X$ containing $x$ such that the $h_i$ are (images of elements) in $B.$ Let $\mathfrak{n}\in \mathrm{MaxSpec}(B)$ be the maximal ideal corresponding to the point $x,$ and let $J := (h_1,\ldots,h_m) B$ be the ideal in $B$ generated by the $h_i$. 

We claim first that $J B_\mathfrak{n}$ is a minimal prime ideal of $B_\mathfrak{n}.$ To see this, note that since $B_\mathfrak{n}\hookrightarrow \mathcal{O}_{X,x}$ is a faithfully flat map, (as these local rings have the same completions), $JB_\mathfrak{n}$ is the contraction of $J\mathcal{O}_{X,x} = \mathfrak{q}$ (see \cite[\href{https://stacks.math.columbia.edu/tag/05CK}{Tag 05CK}]{stacks-project}) and is therefore a prime ideal. Moreover, since $B_\mathfrak{n}\hookrightarrow \mathcal{O}_{X,x}$ is faithfully flat, it has the Going-Down property (\cite[\href{https://stacks.math.columbia.edu/tag/00HS}{Tag 00HS}]{stacks-project}). Therefore, as $\mathfrak{q}$ is a minimal prime of $\mathcal{O}_{X,x}$, its contraction $JB_\mathfrak{n}$, must also be minimal. 

We have $\left(B_\mathfrak{n}/JB_\mathfrak{n}\right)\widehat{} = \widehat{B_\mathfrak{n}}/J\widehat{B_\mathfrak{n}} = \widehat{\mathcal{O}}_{X,x}/\mathfrak{q}\widehat{\mathcal{O}}_{X,x} = \left(\mathcal{O}_{X,x}/\mathfrak{q}\right)\widehat{}.$ Hence, \[\dim(B_\mathfrak{n}/JB_\mathfrak{n}) = \dim(\mathcal{O}_{X,x}/\mathfrak{q}).\] Now let $\mathfrak{p}\subseteq B$ denote the contraction of $JB_\mathfrak{n}$ to $B,$ so then $\mathfrak{p}$ is a minimal prime of $B$ contained in $\mathfrak{n}.$ Then $\dim(B_\mathfrak{n}/JB_\mathfrak{n}) = \dim\left((B/\mathfrak{p}B)_\mathfrak{n}\right) = \dim(B/\mathfrak{p}B),$ where the last equality follows from the fact that $B/\mathfrak{p}$ being an affinoid algebra that is an integral domain, all its maximal ideals have the same height (see \autoref{integral-rigid-spaces-are-equidimensional}). 

Therefore, in all we have shown that $\dim(B/\mathfrak{p}) = \dim(\mathcal{O}_{X,x}/\mathfrak{q}),$ for a minimal prime $\mathfrak{p}$ of $B.$ And since we are assuming that $\dim(\mathcal{O}_{X,x}/\mathfrak{q}) < \dim(X),$ this means that $\dim(B/\mathfrak{p}) < \dim(X).$ However, now find a closed point $\mathfrak{n}' \in \mathrm{MaxSpec}(B)$ containing $\mathfrak{p}$ but not containing any other minimal prime of $B$ (this is possible since $B$ is Jacobson and so closed points are dense). If $\mathfrak{n}'$ corresponds to the point $x' \in X,$ we have $\dim(X) > \dim(B/\mathfrak{p}) \geq \dim(B_{\mathfrak{n}'}) = \dim(\mathcal{O}_{X,x'}).$ This contradicts the equidimensionality of $X.$
\end{proof}

\section{The non-archimedean definable Chow theorem}\label{proof}

The goal of this section is to prove a version of the Definable Chow theorem in the non-archimedean setting. Let $K$ be as in the previous chapter. Namely, $K$ shall denote an algebraically closed field, complete with respect to a non-trivial non-archimedean absolute value. Moreover we assume that $K$ is second countable. The main goal of this chapter is to prove the following result:

\begin{theorem}\label{affine-chow}
Let $X$ be a closed analytic subset of $\mathbb{A}^{n,\text{an}}_{K}$. Suppose that for some tame structure on $K$, $X$ is definable as a subset of $\mathbb{A}^n(K) = K^n.$ Then $X$ is algebraic i.e. $X$ is the vanishing locus of a finite collection of polynomials in $K[t_1,\ldots,t_n].$
\end{theorem}

We outline the major steps of the proof below:
\begin{description}
    \item[Step 0:] Our first step is to show that for a reduced variety $X$ over $K,$ a global analytic function $f \in H^0(X^\mathrm{an},\mathcal{O}_{X^\mathrm{an}})$ whose graph is definable, must be algebraic. This is the content of \autoref{upgraded-polynomial-lemma-v2}. The 
    Proposition may be seen as a non-archimedean definable analogue of Liouville's theorem from complex analysis. The proof proceeds by a devissage argument:
    \begin{itemize}
        \item First, when $X =\mathbb{A}^n_{K}$ - \autoref{polynomial-lemma}.
        \item Second, when $X$ is a smooth affine variety over $K$ - \autoref{upgraded-polynomial-lemma}, using Noether normalization to reduce to the first case. 
        \item And lastly, for a general reduced variety $X$ ( \autoref{upgraded-polynomial-lemma-v2}), using \autoref{rigid-function-algebraic-on-dense-open-implies-algebraic} to reduce to the smooth case.
    \end{itemize}
    \item[Step 1:] Now suppose $X \subseteq \mathbb{A}^{n,\mathrm{an}}_{K}$ is as in the statement of \autoref{affine-chow}. We shall induct on $\dim(X)+n$. By the Theorem of the Boundary, $\dim(\mathrm{Fr}(X)) <\dim(X)$ and so we can find a point $q \in \mathbb{P}^n(K)\setminus \mathrm{cl}(X).$
    \item[Step 2:] The projection from $q$ onto a hyperplane $\mathbb{H}\subseteq K^n$ not containing it, $\pi\vert_X : X \rightarrow \mathbb{H}$ is \emph{finite}. The image $Y = \pi(X)$ is an analytic subset of $\mathbb{A}^{n-1,\mathrm{an}}_{K}$, and by induction therefore algebraic. 
    \item[Step 3:] The \'etale locus $U \subseteq Y$ of $\pi\vert_X :X \rightarrow Y$ is definable (thanks to \autoref{etale-locus}), and of smaller dimension, therefore algebraic.  
    \item[Step 4:] The characteristic polynomial of the finite \'etale map \[\pi :\pi^{-1}(U^\mathrm{an}) \rightarrow U^\mathrm{an},\] has coefficients in $H^0(U^\mathrm{an})$ that are \emph{definable}. By Step 0, we shall then conclude $\pi^{-1}(U^\mathrm{an})\subseteq X$ is algebraic. The complement in $X$ is of smaller dimension, and by induction thus algebraic.
\end{description}

\subsection{A non-archimedean definable Liouville theorem}

\begin{lemma}\label{polynomial-lemma}
Let $(X,\mathcal{O}_X) = \mathbb{A}^{n,an}_{K}$ be the rigid $n$-dimensional affine plane over $K$ and let $f \in H^0(X,\mathcal{O}_X)$ be a global analytic function. Suppose $f$ viewed as a function $f : K^n \rightarrow K$ is definable. Then $f$ is a polynomial function.
\end{lemma}
\begin{proof}
We prove this by induction on $n.$ 

\emph{Case of} $n=1$: A function $f \in H^0(\mathbb{A}^{1,\mathrm{an}}_{K},\mathcal{O}_{\mathbb{A}^{1,\mathrm{an}}_{K}})$ is given by a globally convergent power series $f(z)= \sum_{i \geq 0} a_i z^i.$ Thus, $\lim_{i \rightarrow \infty} (p^{ir}\cdot |a_i|) = 0$ for every $r \geq 0.$ 
For a given $r\geq 0$, the number of zeroes of $f(z)$ on the disk $\{z\in K : |z| \leq p^r\}$ is the number of zeroes (with the same multiplicities) of $g_r(t) := f(p^{-r}t) = \sum_{ \geq 0} a_i p^{-ir} t^i$ in the unit disk $|t|\leq 1,$ which by \autoref{zeroes-of-tate-functions-new} is at least $\epsilon(g_r).$ Now given any $i < j$ with $a_i, a_j \neq 0,$ we note that for $r$ large enough $p^{rj}|a_j|\geq p^{ri}|a_i|$ and thus $\epsilon(g_r) \geq j$. Thus, if $a_i \neq 0$ for infinitely many $i, f$ must have infinitely many zeroes. However, as $f$ is definable, $f^{-1}(0)$ is a definable subset of $K$ that is discrete, and must therefore be necessarily finite. Hence, it cannot be the case that $a_i \neq 0$ for infinitely many $i$, i.e. $f$ is a polynomial.

\emph{Proof for general} $n \geq 1$: The global analytic function $f\in H^0(\mathbb{A}^{n,an}_{K})$ is again given by a globally convergent power series on $K^n.$ Thus, we write $f(z_1,\ldots,z_n) = \sum_{i \geq 0} a_i(z_1,\ldots,z_{n-1}) z_n^i,$ where $a_i \in H^0(\mathbb{A}^{n-1,an}_{K}).$ Moreover, each $a_i(z_1,\ldots,z_{n-1})$ is also definable viewed as a function on $K^{n-1}$ (by \autoref{basic-properties}.(iv)). By induction, we have that the $a_i(z_1,\ldots,z_{n-1})$ are polynomials in $K[z_1,\ldots,z_{n-1}]$. 
From the $n=1$ case, it must be that for every $\underline{\bm{\lambda}}\in K^{n-1},$ the sequence $a_i(\underline{\bm{\lambda}})$ is eventually $0$. In other words, $K^{n-1} = \bigcup_{j \geq 0} \cap_{i \geq j} V(a_i),$ a countable union of closed subsets of $K^{n-1}.$ By the Baire Category Theorem, this is only possible if for some $j \geq 0, K^{n-1} = \bigcap_{i \geq j}V(a_i)$, i.e. $a_i = 0$ for all $i \geq j$ and hence $f$ is a polynomial. 
\end{proof}


\begin{lemma}\label{upgraded-polynomial-lemma}
Let $X$ be an integral, smooth scheme of finite type over $K$ and denote by $X^\text{an}$ the rigid analytification of $X.$ Let $f \in H^0(X^\text{an},\mathcal{O}_{X^{\text{an}}})$ be a global rigid analytic function on $X^\text{an}$ such that the graph of $f$ viewed as a subset of $X(K) \times K$ is definable. Then $f \in H^0(X,\mathcal{O}_X).$ 
\end{lemma}
\begin{proof}
By passing to a \emph{finite} affine cover of $X$ we may assume $X= \text{Spec}(A)$ for a domain $A$ that is regular and a $K$-algebra of finite type. Choose a generically \'etale, Noether normalization of $A$, i.e. a generically \'etale, finite inclusion $i : K[t_1,\ldots,,t_d] \hookrightarrow A.$ If $K$ has characteristic zero, generic \'etaleness comes for free, and when $K$ has positive characteristic that there does exist such a Noether normalization (up to further passing to a finite affine cover) follows for instance from \cite[Theorem 1]{kedlaya2004etale}.

Since $A$ is regular (in particular Cohen-Macaulay), $i$ is locally free by Hironaka's Miracle Flatness criterion \cite[\href{https://stacks.math.columbia.edu/tag/00R4}{Tag 00R4}]{stacks-project}. There is a finite set of polynomials $p_i(\underline{t}), 1\leq i \leq m$ generating the unit ideal in $K[\underline{t}]$ such that $A[p_i^{-1}]$ is free over $K[\underline{t}][p_i^{-1}]$ for each $i$. Moreover, it is easy to see that $K[\underline{t}][p_i^{-1}]$ is finite free over another pure polynomial subring in $d$ variables (just need to change variables). Thus, by replacing $A$ with $A[p_i^{-1}]$ and modifying the Noether normalization map as above, we are in the case where $A$ is finite free over the polynomial ring $K[t_1,\ldots,t_d],$ say of rank $r$.

Let $a_1,\ldots, a_r \in A$ be a module-basis over $K[t_1,\ldots,t_d].$ It follows 
that $H^0(X^\text{an},\mathcal{O}_{X^\text{an}})$ is a free module over $H^0\left(\mathbb{A}^{d,\text{an}}\right)$ again with basis $a_1,\ldots,a_r.$ Thus, $f$ can be written uniquely as $f = \sum_{1\leq k \leq r} a_k\cdot g_k(\underline{t})$ with $g_k(\underline{t}) \in H^0\left(\mathbb{A}^{d,\text{an}}\right).$ To finish the proof, it suffices to show that the $g_k(\underline{t})$ have definable graphs in $K^{d+1},$ since then we may appeal to the previous \autoref{polynomial-lemma} to conclude that the $g_k$ are polynomials. By continuity, it in fact suffices to show that $g_k(\underline{t})\vert_{U}$ has a definable graph for some Zariski dense open subset $U$ of $K^d.$ Since the Noether normalization map $i : \text{Spec}(A)\rightarrow \mathbb{A}^d_{K}$ is generically \'etale, we may take $U \subseteq \mathbb{A}^d_{K}$ to be the locus over which it is \'etale. For any point $\underline{u} \in U,$ letting $i^{-1}(\underline{u}) = \{P_1,\ldots, P_r\},$ we have $r$ linear equations in $r$ variables: \[f(P_j) = \sum_{1\leq k \leq r} a_k(P_j)g_k(\underline{u})\] for each $j \in \{1,\cdots, r\}.$ Over the \'etale locus the matrix $\left(a_k(P_j)\right)_{1\leq k,j\leq r}$ is invertible and thus we may write for each $k,$ the function $g_k(u)$ as an explicit linear combination of $\{f(P_j):1\leq j \leq r\},$ with coefficients being rational function in $a_k(P_j).$ Note that permuting the ordering of the $P_j,$ leaves the specific linear combination invariant. Thus, the graph of the function $g_k :U \rightarrow K$ can be expressed as a first-order formula with all its terms using definable functions and sets-- indeed, we note that $f$ is definable by assumption, and that $U, X(K), i, a_j$ being algebraic are also definable. We thus obtain that the $g_k(\underline{u})$ are definable over $U$, concluding the proof. 

\end{proof}

\begin{lemma}\label{rigid-function-algebraic-on-dense-open-implies-algebraic}
Let $A$ be a reduced finite-type $K$-algebra and let $X = \Spec(A).$ Let $\{X_i\}$ denote the set of irreducible components of $X,$ given their reduced induced structures. Suppose $f \in H^0(X^\text{an},\mathcal{O}_{X^{\text{an}}})$ is a global rigid analytic function such that for every $i,$ there is a dense open subset $U_i \subseteq X_i$ such that $f\vert_{U_i^\text{an}} \in H^0(U_i, \mathcal{O}_{U_i}).$ Then $f \in H^0(X,\mathcal{O}_X).$

\end{lemma}
\begin{proof}
By our assumptions on $f,$ we may view $f$ as an element of the total ring of fractions $Q(A)$ of $A.$ To show that $f \in A,$ it suffices to show that for every maximal ideal $\mathfrak{m}$ of $A,$ the image of $f$ in $Q(A)_\mathfrak{m}$ (the localization of $Q(A)$ at the multiplicative set $A\setminus \mathfrak{m}$) is also in $A_\mathfrak{m}.$ Indeed, writing $f = a/s$ with $a,s \in A$ and $s$ a non-zerodivisor, if $f \notin A,$ then $a \notin sA.$ So we may choose a maximal ideal $\mathfrak{m}$ containing $(sA : a) = \{b \in A : ba \in sA\}.$ However, for this choice of $\mathfrak{m}, f \notin A_\mathfrak{m}.$ 

So let us now fix a maximal ideal $\mathfrak{m}$ of $A$. We note that $Q(A)_\mathfrak{m}$ is in fact the total ring of fractions $Q(A_\mathfrak{m})$ of $A_\mathfrak{m}.$ We also note that since $f \in H^0(X^\text{an}, \mathcal{O}_{X^\text{an}}),$ in particular, $f \in \widehat{A_\mathfrak{m}},$ since $\widehat{A_\mathfrak{m}} = \widehat{\mathcal{O}}_{X^\text{an}, \mathfrak{m}}$ (\cite[Lemma 5.1.2 (2)]{conrad1999irreducible}). For notational simplicity let $B := A_\mathfrak{m}, L := Q(B)$ and $\widehat{L} := L \otimes_{B} \widehat{B}.$ We have \emph{inclusions} $L \subseteq \widehat{L}$ and $\widehat{B} \subseteq \widehat{L}$ and $f \in L \cap \widehat{B}.$ We must show that $f \in B.$ Since $B \subseteq \widehat{B}$ is faithfully flat it suffices to show that $f \otimes 1 = 1\otimes f$ in $\widehat{B}\otimes_B \widehat{B}$ \cite[\href{https://stacks.math.columbia.edu/tag/023M}{Tag 023M}]{stacks-project}. Since $f \in L,$ the equality $f\otimes 1 = 1\otimes f$ evidently holds in $\widehat{L}\otimes_L \widehat{L}$ and we further note that $\widehat{B}\otimes_B \widehat{B}$ \emph{injects} into $(\widehat{B}\otimes_B\widehat{B})\otimes_B L = \widehat{L}\otimes_L \widehat{L}$ since $B$ injects into $L$ and $\widehat{B}\otimes_B \widehat{B}$ is $B$-flat. Hence $f \otimes 1 = 1 \otimes f$ holds in $\widehat{B} \otimes_B \widehat{B},$ as was to be shown.

\end{proof}

\begin{theorem}[A non-archimedean definable Liouville theorem]\label{upgraded-polynomial-lemma-v2}
Let $X$ be a reduced scheme of finite type over $K$ and denote by $X^\text{an}$ the rigid analytification of $X.$ Let $f \in H^0(X^\text{an},\mathcal{O}_{X^{\text{an}}})$ be a global rigid analytic function on $X^\text{an}$ such that the graph of $f$ viewed as a subset of $X(K) \times K$ is definable. Then $f \in H^0(X,\mathcal{O}_X).$
\end{theorem}
\begin{proof}
Again, by passing to a finite affine open cover we may assume that $X$ is affine. For each irreducible component $X_i$ of $X$, let $U_i \subseteq X_i$ be a dense open subset of $X_i$ that is \emph{smooth} over $K$. The restriction $f\vert_{U_i^\text{an}} \in H^0({U_i}^\text{an},\mathcal{O}_{{U_i}^\text{an}})$ is definable and hence by \autoref{upgraded-polynomial-lemma} we have that $f\vert_{U_i^\text{an}} \in H^0(U_i,\mathcal{O}_{U_i}).$ From \autoref{rigid-function-algebraic-on-dense-open-implies-algebraic} we conclude that $f \in H^0(X, \mathcal{O}_X).$ 

\end{proof}

\begin{remark}
It is clear that reducedness of the underlying variety $X$ is necessary in the above \autoref{upgraded-polynomial-lemma-v2} since the graph of a function on the underlying $K$-points doesn't record the nilpotent structure.
For example, take $X = \mathbb{A}^1_{K}[\epsilon] = \Spec(K[t,\epsilon]/(\epsilon^2)).$ Choose any function $g\in H^0(X^\text{an})$ which is not in $H^0(X)$ and take $f = \epsilon \cdot g.$ 
\end{remark}

\subsection{Proof of the non-archimedean Definable Chow theorem}

We now turn towards proving our main \autoref{affine-chow}:
\begin{theorem}
Let $X$ be a closed analytic subset of $\mathbb{A}^{n,\text{an}}_{K}$ that is also definable as a subset of $\mathbb{A}^n_{K}= K^n.$ Then $X$ is an algebraic subset i.e. $X$ is defined as the vanishing locus of a finite collection of polynomials in $K[t_1,\ldots,t_n].$
\end{theorem}
\begin{remark}
Recall that if $\mathscr{A}$ is a rigid analytic space over $K,$ then by a \emph{closed analytic subset} $X\subseteq \mathscr{A}$ we mean that there is a closed immersion of rigid spaces $\mathscr{X} \xhookrightarrow{i} \mathscr{A}$ such that $i(\mathscr{X}) = X.$ Equivalently, $X$ is cut out by the vanishing locus of a coherent $\mathcal{O}_{\mathscr{A}}$-ideal,  or more concretely, there is an admissible affinoid covering $\mathscr{A} = \displaystyle\cup_{i \in I} U_i,$ and for each $i \in I,$ finitely many functions $f^{(i)}_1,\ldots, f^{(i)}_{n(i)}$ in $\mathcal{O}_\mathscr{A}({U_i})$ such that $X\cap U_i$ is the vanishing locus of $\{f^{(i)}_1,\ldots, f^{(i)}_{n(i)}\}.$ Moreover, we note that given a closed analytic subset $X \subseteq \mathscr{A}$ as above, there is a canonical structure of a reduced rigid analytic space that can be put on $X,$ with a canonical closed immersion $X \hookrightarrow \mathscr{A}$ (see \cite[\S 9.5.3, Proposition 4]{bosch1984non}). We shall refer to this reduced structure as the \emph{reduced induced structure} on $X.$  
\end{remark}
As was outlined earlier, the proof of the theorem shall proceed by inducting on the dimension of the definable set $X\subseteq K^n,$ (which agrees with the dimension of $X$ as an analytic space - \autoref{rigid-dimension-equals-definable-dimension}). 
First, we prove a preparatory Lemma concerning the \'etale locus of a finite morphism of rigid varieties that shall be used in the proof. 

\begin{lemma}\label{etale-locus}
Suppose $\pi : X \rightarrow Y$ is a finite surjective morphism of reduced rigid analytic varieties over $K.$ Suppose $X$ is equidimensional at every point (i.e. for all $x \in X, \, \mathrm{dim}(\mathcal{O}_{X,x}) = \mathrm{dim}(X)$) and suppose $Y$ is irreducible and normal (i.e. for all $y \in Y, \mathcal{O}_{Y,y}$ is a normal domain). Let $N$ be the generic fiber cardinality of $\pi$ (i.e. $N = \mathrm{rank}_{\mathcal{O}_Y}(\pi_*\mathcal{O}_X)$). Then for $y \in Y$, $\pi$ is \'etale at every point in the fiber of $y$ if and only if $|\pi^{-1}(y)| = N.$
\end{lemma}
\begin{remark}
\begin{enumerate}[label=(\alph*).]
    \item We recall that a morphism of rigid spaces $\pi : X\rightarrow Y$ is said to be \'etale at  a point $x\in X$ iff the induced map on local rings $\mathcal{O}_{Y,\pi(x)} \rightarrow \mathcal{O}_{X,x}$ is flat and unramified (see \cite[\S 3]{dejong1996etale}).
    \item When we say the generic fiber cardinality is $N$ we mean that for every $y \in Y,$ we have $N = \dim_{Q(\mathcal{O}_{Y,y})}\left((\pi_*\mathcal{O}_X)_y\otimes_{\mathcal{O}_{Y,y}} Q(\mathcal{O}_{Y,y})\right)$. Here $Q(\mathcal{O}_{Y,y})$ denotes the fraction field of the domain $\mathcal{O}_{Y,y}.$ Since $Y$ is connected, the dimension on the right hand side is indeed independent of the point $y.$ To see this, it suffices to work over a connected affinoid open $\Sp(A)$ of $Y.$ Then $A$ must be a domain, and since $\pi$ is a finite map, $\pi^{-1}(\Sp(A))$ is an affinoid open $\Sp(B)$ of $X$ with the induced map $A\rightarrow B$ making $B$ a finite $A$-module. For a point $y \in \Sp(A)$ corresponding to the maximal ideal $\mathfrak{m}$ of $A,$ we have that $\dim_{Q(\mathcal{O}_{Y,y})}\left((\pi_*\mathcal{O}_X)_y\otimes_{\mathcal{O}_{Y,y}} Q(\mathcal{O}_{Y,y})\right) = \dim_{Q(\mathcal{O}_{Y,y})} \left(B\otimes_A Q(\mathcal{O}_{Y,y}) \right)$ $= \dim_{Q(A)} B\otimes_A Q(A).$ 
\end{enumerate}\end{remark}
\begin{proof}[Proof of \autoref{etale-locus}] By working locally over connected affinoid opens of $Y$, we may assume that $Y = \Sp(A)$ is affinoid. Since $\pi$ is a finite morphism, $X$ is also affinoid, and $X = \Sp(B)$ with the induced map $A\rightarrow B$ making $B$ a finite $A$-module. The assumptions on $Y$ imply that $A$ is a normal integral domain. Let $F$ denote its fraction field and let $N = \dim_F(B\otimes_A F)$ be the generic fiber cardinality of $\pi.$ 

For a point $x \in X,$ if we denote the maximal ideals corresponding to $x, \pi(x)$ by $\mathfrak{n}\subseteq B, \mathfrak{m}\subseteq A$ respectively, then we note that since $\widehat{A}_\mathfrak{m} = \widehat{\mathcal{O}}_{Y,\pi(x)}$ and $\widehat{B}_\mathfrak{n} = \widehat{\mathcal{O}}_{X,x}$ the map $\mathcal{O}_{Y,\pi(x)}\rightarrow \mathcal{O}_{X,x}$ is flat and unramified if and only if the same holds for the map $A_\mathfrak{m} \rightarrow B_\mathfrak{n}$ (the fact that both maps are unramified simultaneously is easy to see, whereas for flatness one may use the local flatness criterion \cite[Theorems 22.1 and 22.4]{matsumura1989commutative}).

Suppose now $y \in Y$ is a point corresponding to the maximal ideal $\mathfrak{m}$ of $A$ such that $\pi$ is \'etale at every point of $\pi^{-1}(y).$ Then from the above $B/\mathfrak{m}B$ must be unramified over $A/\mathfrak{m}$ and thus $|\pi^{-1}(y)| = \dim_{A/\mathfrak{m}} B/\mathfrak{m}B.$ Similarly, it folows that $B\otimes_A A_\mathfrak{m}$ is finite flat(hence free) over $A_\mathfrak{m}$ and hence $\mathrm{rank}_{A_\mathfrak{m}}(B\otimes_A A_\mathfrak{m}) = \dim_{A/\mathfrak{m}} B/\mathfrak{m}B = \dim_K(B\otimes_A F) = N.$ Therefore, we see that $|\pi^{-1}(y)| = N.$

Before turning to prove the converse direction, we first show that $\dim(X) = \dim(Y).$ By \autoref{affinoid's-dimension}, $\dim(X) = \dim(B)$ and $\dim(Y) = \dim(A).$ Since $\pi : \mathrm{MaxSpec}(B) \rightarrow \mathrm{MaxSpec}(A)$ is surjective, the image of $\Spec(B) \rightarrow \Spec(A)$ contains all the closed points of $\Spec(A).$ If $I$ denotes the kernel of $A\rightarrow B,$ then $A/I \hookrightarrow B$ is a finite \emph{inclusion} of rings, and so by the Going Up Theorem, the image $\Spec(B)\rightarrow \Spec(A)$ is the vanishing locus $V(I)$. Thus, we must have $V(I)\supseteq \mathrm{MaxSpec}(A)$. However since $A$ is Jacobson (by \cite[\S 5.2.6, Theorem 3]{bosch1984non}), this implies $V(I) = \Spec(A)$ and hence $I = {0}.$ Thus, $A\hookrightarrow B$ is a finite inclusion of rings and therefore $\dim(A) = \dim(B).$

Now suppose $y \in Y$ is a point such that $|\pi^{-1}(y)| = N.$ Let $\pi^{-1}(y) = \{x_1,\ldots, x_N\}.$ We would like to show that $\pi$ is \'etale at each $x_i.$ We have a canonical isomorphism $B \otimes_A \mathcal{O}_{Y,y}=(\pi_*\mathcal{O}_X)_y \cong \prod_{i=1}^N \mathcal{O}_{X,x_i}$ (see \cite[pp. 481-482]{conrad1999irreducible}). If $L$ denotes the fraction field of $\mathcal{O}_{Y,y},$ we have that $B\otimes_A L \cong \prod_{i=1}^N \mathcal{O}_{X,x_i}\otimes_{\mathcal{O}_{Y,y}} L$. 

\emph{Subclaim: For each $i$, the natural map $\mathcal{O}_{X,x_i}\rightarrow \mathcal{O}_{X,x_i}\otimes_{\mathcal{O}_{Y,y}} L$ is injective.}  
\begin{proof}[Proof of Subclaim:]
Note that $\mathcal{O}_{X,x_i}\otimes_{\mathcal{O}_{Y,y}} L$ is the localisation of $\mathcal{O}_{X,x_i}$ at the (image inside $\mathcal{O}_{X,x_i}$ of the) multiplicative set $\mathcal{O}_{Y,y}\setminus\{0\}.$ Thus, the claim is equivalent to showing that $\mathcal{O}_{X,x_i}$ is a torsion-free $\mathcal{O}_{Y,y}$-module. Equivalently, we must show that the image of $\mathcal{O}_{Y,y}\setminus\{0\}$ inside $\mathcal{O}_{X,x_i}$ does not contain any zero-divisors of the ring $\mathcal{O}_{X,x_i}.$ Since $\mathcal{O}_{X,x_i}$ is a reduced ring, the set of zero-divisors of $\mathcal{O}_{X,x_i}$ is the union of the minimal prime ideals of $\mathcal{O}_{X,x_i}$ (\cite[\href{https://stacks.math.columbia.edu/tag/00EW}{Tag 00EW}]{stacks-project}). Therefore, it suffices to prove that every minimal prime ideal $\mathfrak{q}$ of $\mathcal{O}_{X,x_i}$ contracts to the $(0)$-ideal of $\mathcal{O}_{Y,y}.$ If we set $\mathfrak{q}\cap \mathcal{O}_{Y,y} =: \mathfrak{p},$ then $\mathcal{O}_{Y.y}/\mathfrak{p} \hookrightarrow \mathcal{O}_{X,x_i}/\mathfrak{q}$ a finite \emph{inclusion} of domains and hence, $\dim(\mathcal{O}_{Y.y}/\mathfrak{p}) = \dim(\mathcal{O}_{X,x_i}/\mathfrak{q}).$ We now have the chain of equalities: \[\dim(\mathcal{O}_{Y.y}/\mathfrak{p}) = \dim(\mathcal{O}_{X,x_i}/\mathfrak{q}) = \dim(X) = \dim(Y) = \dim(\mathcal{O}_{Y,y})\]
where the second equality is from \autoref{local-ring-is-equidimensional}, the third from the previous paragraphs and the last from \autoref{integral-rigid-spaces-are-equidimensional}. But since $\mathcal{O}_{Y,y}$ is an integral domain the equality $\dim(\mathcal{O}_{Y.y}/\mathfrak{p})=\dim(\mathcal{O}_{Y.y})$ is only possible if $\mathfrak{p} = (0)$. This completes the proof of the subclaim.
\end{proof} 


The subclaim shows in particular, that $\mathcal{O}_{X,x_i}\otimes_{\mathcal{O}_{Y,y}} L$ must be non-zero for each $i$. But since, $\dim_L(B\otimes_A L) = N,$ this is only possible if $L = \mathcal{O}_{X,x_i}\otimes_{\mathcal{O}_{Y,y}} L$ for each $i.$ In particular, $\mathcal{O}_{X,x_i}\subseteq L.$ However, since $\mathcal{O}_{Y,y}$ is a normal domain, and since $\mathcal{O}_{X,x_i}$ is finite over $\mathcal{O}_{Y,y},$ we get that $\mathcal{O}_{Y,y} = \mathcal{O}_{X,x_i},$ and so $\pi$ is evidently \'etale at $x_i.$
\end{proof}

\begin{proof}[Proof of \autoref{affine-chow}]
We induct on $d+n$ where $d := \mathrm{dim}(X)$. If $d = 0,$ then $X$ is finite, hence algebraic. And if $d = n,$ then $X = K^n$ and we're done.

Suppose that $n > d \geq 1.$ It follows from \autoref{rigid-dimension-equals-definable-dimension} that the set $S = \{x \in X: \dim(\mathcal{O}_{X,x}) < d\}$ is a \emph{definable} subset of $K^n.$ Moreover, the closure of $S$ in $K^n, \mathrm{cl}(S)$ is the union of the irreducible components of $X$ of dimension $< d,$ and so by the induction hypothesis, $\mathrm{cl}(S)$ is an algebraic subset of $K^n.$ It suffices to then show that $X\setminus S$ is an algebraic subset of $K^n.$ Note that $X\setminus S$ is the union of the irreducible components of dimension $d$ and therefore $X \setminus S$ is a closed, analytic subset of $K^n$. Thus, replacing $X$ by $X\setminus S$ we may assume that $X$ is \emph{equidimensional} of dimension $d.$ 

\emph{Finding a point $q \in \mathbb{P}^n(K)\setminus \mathrm{cl}(X)$:} Embed $K^n \subseteq \mathbb{P}^n(K)$ inside projective $n$-space and denote the homogeneous coordinates of $\mathbb{P}^n(K)$ by $Z_1,\ldots, Z_{n+1}$. Let $\mu$ denote the point $[1:0:\ldots:0] \in \mathbb{P}^n(K)\setminus K^n,$ and consider the neighbourhood $\Delta:= \{|Z_1|\geq |Z_{2}|, \ldots, |Z_1|\geq |Z_{n+1}|\} \subseteq \mathbb{P}^n(K)$ of the point $\mu.$ The neighbourhood $\Delta$ is naturally homeomorphic to the closed unit $n$-dimensional disk, $(K^\circ)^{n},$ via the map $\varphi : \Delta \rightarrow (K^\circ)^n$ that sends $[Z_1:\ldots:Z_{n+1}] \mapsto \left(\frac{Z_2}{Z_1},\ldots,\frac{Z_{n+1}}{Z_1}\right)$ and $S := \varphi(X\cap \Delta)$ is a definable subset of $(K^\circ)^n$ of dimension $\leq d$ contained in $(K^\circ)^{n-1}\times K^\circ\setminus\{0\}.$ We note that since $\mathrm{cl}(S)\cap (K^\circ)^{n-1}\times \{0\} \subseteq \mathrm{Fr}(S),$ and from the Theorem of the Boundary (\autoref{theorem-of-boundary}), since $\dim(\mathrm{Fr}(S)) < d \leq (n-1)$ we can find a point $q \in (K^\circ)^{n-1}\times \{0\}$ such that $q \notin \mathrm{cl}(S),$ and pulling back via $\varphi$ to $\Delta,$ we find a point $q \in \mathbb{P}^n(K)\setminus K^n$ such that $q \notin \mathrm{cl}_{\mathbb{P}^n(K)}(X).$
The point $q \in \mathbb{P}^n(K)\setminus K^n = \mathbb{P}^{n-1}(K)$ defines a line in $K^n.$ 
Now let $T \subseteq X(K)$ be a countable subset of $X(K)$ consisting of \emph{smooth} points of $X$ such that the only closed analytic subset of $X$ containing $T$ is $X$ itself. By the Baire category theorem we may further assume that for every $t \in T$, the line defined by the point $q$ is not contained in the tangent space $T_t(X)$. 

Consider any $(n-1)$-dimensional linear subspace $\mathbb{H} \subseteq K^n,$ not containing the line defined by $q,$ and let $\pi : K^n \rightarrow \mathbb{H}$ denote the projection onto $\mathbb{H}$ with kernel being the line defined by $q.$ We are free to make a linear change of coordinates on $K^n,$ and so we may even assume for simplicity that $q = [0:\ldots:0:1] \in \mathbb{P}^{n-1}(K)$ and that $\pi : K^n \rightarrow K^{n-1}$ is the projection to the first $(n-1)$-coordinates. 
\begin{lemma}
The projection $\pi \vert_{X} : X \rightarrow K^{n-1}$ is a finite morphism of rigid analytic spaces (endowing $X$ with the reduced induced structure).
\end{lemma}
\begin{proof}
\emph{$\pi\vert_X$ is quasi-finite:} Indeed, for $\underline{z} \in K^{n-1},$ $\pi^{-1}(\underline{z})\cap X$ is a closed analytic subset of the 1-dimensional line $\pi^{-1}(\underline{z})$ and is in addition definable. If $\dim(\pi^{-1}(\underline{z})\cap X) = 1,$ then $\pi^{-1}(\underline{z}) \subseteq X,$ which would imply that $q = [0:\ldots:0:1] \in \mathrm{cl}_{\mathbb{P}^n(K)}(X)$ contradicting our choice of $q.$ Thus, $\dim(\pi^{-1}(\underline{z})\cap X) = 0,$ i.e. $\pi^{-1}(\underline{z})\cap X$ is finite.

To show that $\pi\vert_X$ is a finite morphism, it thus remains to show that $\pi\vert_X$ is a proper morphism of rigid spaces (\cite[\S 9.6.3, Corollary 6]{bosch1984non}). In order to prove this, we consider the map $\pi\vert_X$ on the level of the associated Berkovich spaces. Note that $X$ being a closed analytic subvariety of rigid affine $n$-space, is a quasi-separated rigid space and has an admissible affinoid covering of finite type, and moroeover its associated Berkovich analytic space is `good' in the sense of \cite[Remark 1.2.16 \& \S 1.5]{berkovich1993etale}. Recall that the morphism $\pi\vert_{X} : X^\mathrm{Berk} \rightarrow \mathbb{A}^{n-1,\mathrm{Berk}}_{K}$ of good $K$-analytic spaces is proper if it is topologically proper and boundaryless (or `compact and closed' in the terminology of \cite[pp. 50]{berkovich1990spectral}). 

\emph{$\pi\vert_X$ is separated and topologically proper:} 
$\pi\vert_X$ is indeed separated. If $E(0,\underline{r})$ denotes the closed polydisc of polyradius $\underline{r}$ in $\mathbb{A}^{n-1,\mathrm{Berk}}_{K},$ i.e.  $E(0,\underline{r})=\mathcal{M}(K\{r_1^{-1}T_1, \ldots ,  r_{n-1}^{-1}T_{n-1}\}),$  then we claim that $\pi^{-1}(E(0,\underline{r}))\cap X^\mathrm{Berk}$ is bounded in $\mathbb{A}^{n, \mathrm{Berk}}_{K}.$ If it weren't, there would be a sequence of points $x_i \in \pi^{-1}(E(0,\underline{r}))\cap X^\mathrm{Berk}$ with $|T_n(x_i)| \rightarrow \infty$ as $i \rightarrow\infty.$ We may even find a sequence $x_i \in X$ since by \cite[Proposition 2.1.15]{berkovich1990spectral}, the set of rigid points is everywhere dense. But this would again imply that $q = [0:\ldots:0:1] \in \mathrm{cl}_{\mathbb{P}^n(K)}(X)$ contradicting the choice of $q.$ Since every compact subset of $\mathbb{A}^{n-1,\mathrm{Berk}}_{K}$ is contained in some $E(0,\underline{r}),$ it follows that the inverse image of compact sets under the map $\pi\vert_X : X^\mathrm{Berk} \rightarrow \mathbb{A}^{n-1,\mathrm{Berk}}_{K}$ are compact. Thus, $\pi\vert_X$ is topologically proper.

\emph{\ $\pi\vert_X$ is boundaryless:} Since $X^\mathrm{Berk} \hookrightarrow \mathbb{A}^{n,\mathrm{Berk}}_{K}$ is a closed immersion, $\mathrm{Int}(X^\mathrm{Berk}/\mathbb{A}^{n,\mathrm{Berk}}_{K}) = X^\mathrm{Berk}.$ By \cite[Proposition 3.1.3 (ii)]{berkovich1993etale} it suffices to note that $\mathrm{Int}(\mathbb{A}^{n,\mathrm{Berk}}_{K}/\mathbb{A}^{n-1,\mathrm{Berk}}_{K}) = \mathbb{A}^{n,\mathrm{Berk}}_{K}.$ To see this last equality, for any $x\in \mathbb{A}^{n,\mathrm{Berk}}_{K}$ let $y = \pi(x),$ and choose an affinoid neighbourhood $E(0,\underline{r}) \subseteq \mathbb{A}^{n-1,\mathrm{Berk}}_{K}$ containing $y$ in its interior. Choosing an $R \in |K^\times|$ with $R > |T_n(x)|$, we see that $\chi_x : K\{r_1^{-1}T_1,\ldots,r_{n-1}^{-1}T_{n-1},R^{-1}T_n\} \rightarrow \mathcal{H}(x)$ is inner over $K\{r_1^{-1}T_1,\ldots,r_{n-1}^{-1}T_{n-1}\},$ i.e. $x \in \mathrm{Int}(\mathbb{A}^{n,\mathrm{Berk}}_{K}/\mathbb{A}^{n-1,\mathrm{Berk}}_{K}).$ 

Therefore, the map $\pi\vert_X : X^\mathrm{Berk} \rightarrow \mathbb{A}^{n-1,\mathrm{Berk}}_{K}$ is proper and hence by \cite[Proposition 3.3.2]{berkovich1990spectral}, so is $\pi\vert_X : X \rightarrow \mathbb{A}^{n-1,\mathrm{an}}_{K}.$ 
\end{proof}

Since $\pi\vert_X : X \rightarrow \mathbb{A}^{n-1,\mathrm{an}}_{K}$ is finite, the image $Y := \pi(X)$ is a closed analytic subvariety of $K^{n-1},$ by \cite[\S 9.6.3, Proposition 3]{bosch1984non}. In addition, as $Y$ is a definable subset, by the induction hypothesis $Y$ is an algebraic subset of $K^{n-1}.$ Endowing $Y$ with its structure as a reduced closed affine algebraic subvariety of $\mathbb{A}^{n-1}_{K}$, the morphism $\pi\vert_X$ gives rise to a finite, surjective morphism of rigid analytic spaces $\pi\vert_X : X \rightarrow Y^\mathrm{an}.$   

\begin{lemma}\label{etale}
There is a Zariski dense open $U \subseteq Y,$ such that $\pi\vert_X^{-1}(U^\mathrm{an}) \rightarrow U^\mathrm{an}$ is a finite, \'etale surjection of rigid varieties. 
\end{lemma}
\begin{proof}
We first claim that the map $\pi\vert_X : X \rightarrow Y^\mathrm{an}$ is \'etale at at least one point of $X$. If the characteristic of $K$ is zero this is immediate. In the general case, note that $\pi(T) \subseteq Y$ is not contained in any analytic subset of $Y^\mathrm{an}$ and in particular, there is a $t_0 \in T$ such that $\pi(t_0)$ is a smooth point of $Y$. Then note that by our choice of $q$ earlier, we have that the map $\pi\vert_X : X \rightarrow Y^\mathrm{an}$ is injective on the level of tangent spaces $T_{t_0}(X) \hookrightarrow T_{\pi(t_0)}(Y^\mathrm{an}).$ Therefore, applying \cite[Ch 4, Corollary 3.27]{liu2002algebraic} we see that $\pi\vert_X$ is \'etale at $t_0$.

Let $\{Y_i\}_{1\leq i \leq r}$ denote the finitely many irreducible components of $Y,$ thus $\{Y_i^\mathrm{an}\}_{1\leq i \leq r}$ being those of $Y^\mathrm{an}.$ Let $U_i \subseteq Y_i \setminus \bigcup_{j \neq i} Y_j,$ be a non-empty, \emph{principal} open subset of $Y$ so that each $U_i$ is an integral (reduced and irreducible) open subvariety of $Y,$ (hence $U_i^\mathrm{an}$ is a reduced and irreducible admissible open of $Y^\mathrm{an}$ \cite[Theorem 5.1.3 (2)]{conrad1999irreducible}). $U_i$ being a principal open subset of $Y \subseteq K^{n-1}$, may be viewed as a closed affine subvariety of $K^n.$ By \autoref{etale-locus}, the \'etale locus $E_i \subseteq U_i^\mathrm{an}$ of $\pi : \pi^{-1}(U_i^\mathrm{an}) \rightarrow U_i^\mathrm{an}$ is definable as it may be defined using a first-order formula expressing $E_i$ as the subset of points in $U_i^\mathrm{an}$ whose fiber under $\pi$ has cardinality equal to the generic fiber cardinality over $U_i^\mathrm{an}$. Moreover, the complement $U_i^\mathrm{an}\setminus E_i$ is a closed analytic subvariety of $U_i^\mathrm{an}\subseteq K^n$ of dimension $< \dim(U_i) = d.$ By the induction hypothesis, $U_i^\mathrm{an}\setminus E_i$ is a Zariski closed algebraic subset of $U_i$ and hence $E_i$ is a Zariski dense open of $U_i.$ Now setting $U = \cup_i E_i,$ completes the proof of the Lemma.
\end{proof}
 
Let $U$ be as in \autoref{etale} above, let $\{U_j\}$ be the finitely many open connected components of $U$ and let $V_j := \pi\vert_X^{-1}(U_j).$ Suppose the fiber cardinality of $\pi\vert_{V_j} : \pi\vert_X^{-1}(U_j) \rightarrow U_j$ is $N_j.$ The characteristic polynomial of $T_n\vert_{V_j}$ over $U_j$ (here $T_n$ being the last coordinate function of $\C_p^n$) is a polynomial of degree $N_j$ with coefficients in $\mathcal{O}_{Y^\mathrm{an}}(U_j^\mathrm{an}).$ Moreover, since the $U_j$ are Zariski opens and since $X$ is definable it follows that the coefficients are also definable since they may be defined as symmetric polynomials in the fibers of $\pi\vert_X.$ Hence, by \autoref{upgraded-polynomial-lemma-v2} the characteristic polynomial in fact has coefficients in $\mathcal{O}_{Y}(U_j).$ If $W \subseteq K^{n-1}$ is a Zariski open subset such that $W \cap Y = U,$ then it follows from the above that $X \cap (W\times K)$ is a closed algebraic subset of $W \times K.$ 

If we let $Z$ denote the Zariski closure of $X \cap (W\times K)$ in $K^n,$ then $Z$ is also the closure of $X \cap (W \times K)$ in the metric topology of $K^n,$ (\cite[Theorem 5.1.3 (2)]{conrad1999irreducible}) and hence $Z \subseteq X.$ Moreover, $X\setminus (W\times K) = \pi\vert_X^{-1}(Y\setminus U)$ and so $\dim(X\setminus (W\times K)) < \dim(Y)\leq d.$ By the induction hypothesis $X\setminus (W\times K)$ is thus a closed algebraic subset of $K^n$ and since $X = Z \cup (X\setminus (W\times K))$, we get that $X$ is algebraic, finishing the Proof of \autoref{affine-chow}.
\end{proof}

We obtain as a corollary:

\begin{corollary}\label{definable-chow-theorem}
Let $V$ be a reduced algebraic variety over $K,$ and let $X \subseteq V^\mathrm{an}$ be a closed analytic subvariety of the rigid analytic variety $V^\mathrm{an}$ associated to $V$, such that $X \subseteq V(K)$ is definable in a tame structure on $K$. Then $X$ is algebraic.
\end{corollary}

For a \emph{proper} algebraic variety, every closed analytic subvariety is definable in the tame structure of the rigid subanalytic sets. Thus the familiar version of Chow's theorem for proper varieties follows from \autoref{affine-chow}:

\begin{corollary}[Chow's theorem for proper varieties]
Every closed analytic subset of the rigid analytic variety associated to a proper algebraic variety over $K$ is algebraic.
\end{corollary} 


\bibliographystyle{alpha}
\bibliography{chow}

\textsc{Dept. of Mathematics, University of Toronto, Toronto, Canada}

\textit{E-mail address}: \texttt{abhishek@math.toronto.edu}
\end{document}